\documentclass[12pt,a4paper]{amsart}

\usepackage[mathscr]{eucal}
\usepackage{amssymb}
\usepackage{amsxtra}
\usepackage{hyperref, cleveref}
\usepackage{graphics,color,graphicx}
\usepackage[T1]{fontenc}
\usepackage[utf8]{inputenc}

\newcommand{\fF}{\mathfrak{F}}

\newcommand{\fL}{\mathfrak{L}}

\newcommand{\pA}{\mathcal{A}}
\newcommand{\pB}{\mathcal{B}}

\newcommand{\pL}{\mathcal{L}}

\newcommand{\pS}{\mathcal{S}}

\newcommand{\eK}{\mathscr{K}}
\newcommand{\eL}{\mathscr{L}}
\newcommand{\eM}{\mathscr{M}}
\newcommand{\eN}{\mathscr{N}}
\newcommand{\eP}{\mathscr{P}}

\newcommand{\eR}{\mathscr{R}}

\newcommand{\eV}{\mathscr{V}}

\newcommand{\eW}{\mathscr{W}}

\newcommand{\bC}{\mathbb{C}}

\newcommand{\bM}{\mathbb{M}}
\newcommand{\bN}{\mathbb{N}}

\newcommand{\bdB}{\boldsymbol{B}}

\newcommand{\bdD}{\boldsymbol{D}}
\newcommand{\bdE}{\boldsymbol{E}}

\newcommand{\bdI}{\boldsymbol{I}}
\newcommand{\bdJ}{\boldsymbol{J}}

\newcommand{\bdN}{\boldsymbol{N}}

\newcommand{\bdS}{\boldsymbol{S}}
\newcommand{\bdT}{\boldsymbol{T}}
\newcommand{\bdU}{\boldsymbol{U}}
\newcommand{\bdV}{\boldsymbol{V}}
\newcommand{\bde}{\boldsymbol{e}}
\newcommand{\bdf}{\boldsymbol{f}}

\newcommand{\bdx}{\boldsymbol{x}}
\newcommand{\bdy}{\boldsymbol{y}}

\newcommand{\bdo}{\boldsymbol{0}}

\newcommand{\diag}{{\rm diag}}

\DeclareMathOperator{\Lat}{Lat}

\DeclareMathOperator{\Alg}{Alg}
\DeclareMathOperator{\Refl}{Ref}

\DeclareMathOperator{\Col}{Col}

\newcommand{\Int}{{\rm i}}

\newtheorem{theorem}{Theorem}
\newtheorem{proposition}[theorem]{Proposition}
\newtheorem{lemma}[theorem]{Lemma}
\newtheorem{corollary}[theorem]{Corollary}
\theoremstyle{definition}

\newtheorem{example}[theorem]{Example}

\numberwithin{equation}{section}

\begin{document}

\title[Collineations preserving the lattice of invariant subspaces]{Collineations preserving the lattice of  invariant subspaces of a linear transformation}
\author[J. Bra\v{c}i\v{c}]{Janko Bra\v{c}i\v{c}}
\address{Faculty of Natural Sciences and Engineering, University of Ljubljana, A\v{s}ker\v{c}eva c. 12, SI-1000 Ljubljana, Slovenia}
\email{janko.bracic@ntf.uni-lj.si}
\author[M. Kandi\'{c}]{Marko Kandi\'{c}}
\address{Faculty of Mathematics and Physics, University of Ljubljana, Jadranska 19, SI-1000 Ljubljana, Slovenia}
\email{marko.kandic@fmf.uni-lj.si}


\begin{abstract}
Given a linear transformation $A$ on a finite-dimensional complex vector space $\eV$, in this paper we study the group $\Col(A)$ consisting of those invertible linear transformations $S$ on $\eV$ for which the mapping $\Phi_S$ defined as $\Phi_S\colon \eM\mapsto S\eM$ is an automorphism of the lattice $\Lat(A)$ of all invariant subspaces of $A$. By using the primary decomposition of $A$, we first reduce the problem of characterizing $\Col(A)$ to the problem of characterizing the group $\Col(N)$ of a given nilpotent linear transformation $N$. While $\Col(N)$ always contains all invertible linear transformations of the commutant $(N)'$ of $N$, it is always contained in the reflexive cover $\Alg\Lat(N)'$ of $(N)'$. We prove that  $\Col(N)$ is a proper subgroup of $(\Alg\Lat(N)')^{-1}$ if and only if at least two Jordan blocks in the Jordan decomposition of $N$ are of dimension $2$ or more. We also determine the group $\Col(\bdJ_2\oplus \bdJ_2)$.
\end{abstract}
\maketitle

\section{Introduction} \label{sec01}
\setcounter{theorem}{0}

Let $\fL$ be the lattice of all subspaces of a finite-dimensional complex vector space $\eV$ ordered by the set inclusion.
A bijective map $\Phi: \fL \to \fL$ is said to be a projectivity (see \cite[p. 40]{Bae}) if
$$\eM_1\leq \eM_2\iff \Phi(\eM_1)\leq \Phi(\eM_2).$$
Hence, projectivities are precisely automorphisms of the lattice $\fL$.
If $S$ is an invertible semi-linear transformation on $\eV$, then it is easy to see that the mapping $\Phi_S\colon \fL\to \fL$ given by
$\Phi_S(\eM)=S\eM$ defines a projectivity on $\fL$. If a projectivity $\Phi_S$ is induced by a linear transformation $S$,
then it is called a collineation (see \cite[p. 62]{Bae}). If $\dim(\eV)\geq 3$, then by the first fundamental theorem of projective geometry
every projectivity is induced by a semi-linear transformation.
McAsey and Muhly \cite {MM} studied automorphisms of the sublattice $\Lat(A)$ of $\fL$ of all invariant subspaces of a given linear transformation $A$ on $\eV$. They proved in \cite[Theorem 1.1]{MM} that for a linear transformation $A$ satisfying a certain condition every automorphism $\Phi$ of $\Lat(A)$ which is continuous with respect to the gap-topology (for the definition see \cite{DP}) is of the form $\Phi_S$ where $S$ is a
sum of a linear and a conjugate linear transformation on $\eV$.

Inspired by the first fundamental theorem of projective geometry and the result of McAsey and Muhly, in this paper we study those invertible linear transformations $S$ on $\eV$ which induce automorphisms $\Phi_S$ of $\Lat(A)$. More precisely, we are interested in the subgroup
\begin{equation*}
\Col(A)=\{ S\in \pL(\eV)^{-1};\quad \eM\in \Lat(A) \iff S\eM\in \Lat(A)\}
\end{equation*}
of the group $\pL(\eV)^{-1}$ of all invertible linear transformations on $\eV$. Hence, $S\in \Col(A)$ if and only if $\Phi_S$ is a projectivity which satisfies $\Phi_S\bigl(\Lat(A)\bigr)=\Lat(A)$. With a slight abuse of terminology the elements of $\Col(A)$ are called collineations of $A$.

The paper is organized as follows. In \Cref{secPrelim} we introduce basic notions and some results needed throughout
the text. In \Cref{Section01} we prove that every invertible linear transformation in $(A)'$ belongs to $\Col(A)$  and that every cyclic subspace of a nilpotent linear transformation is a cycle in $\Lat(A)$.
In Section \ref{sec02}, for a given linear transformation $A\in \eV$ we reduce the problem of characterizing $\Col(A)$ to the problem of characterizing the group of collineations of a given nilpotent linear transformation by using the primary decomposition of $A$. Here it is worthwile noting that, in general,  $\Col(A)$ is not a direct sum of groups of collineations of primary factors of $A$.
Section \ref{sec03} is devoted to the reflexive cover $\Refl((N)')=\Alg\Lat(N)'$ of the commutant $(N)'$ of a nilpotent linear transformation $N$. The main results of the paper are in
Section \ref{sec04} where we consider collineations of nilpotent linear transformations. \Cref{prop07} together with \Cref{prop08} yields that for a nilpotent linear transformation $N$ we always have
$$\left((N)'\right)^{-1}\subseteq \Col(N)\subseteq (\Alg\Lat(N)')^{-1}.$$ In \Cref{theo01} we prove that the inclusion $\Col(N)\subseteq (\Alg\Lat(N)')^{-1}$ is strict if and only if at least two Jordan blocks in the Jordan decomposition of $N$ are of dimension $2$ or more.
In the last section we
determine $\Col(\bdJ_2\oplus \bdJ_2)$ where $\bdJ_2$ denotes the $2\times 2$ nilpotent Jordan Block. This is the simplest case when the group $\Col(N)$ is a proper subgroup of $\bigl(\Alg\Lat(N)'\bigr)^{-1}$.

\section{Preliminaries}\label{secPrelim}

Let $\eV$ be a finite-dimensional complex vector space and let $\fL$ be the family of all subspaces of $\eV$. The set inclusion on $\fL$ induces a partial ordering defined by
$\eM_1\leq \eM_2$ if and only if $\eM_1\subseteq \eM_2$. Then $(\fL,\leq)$ is a lattice where the lattice operations are given by
$$\eM_1\vee \eM_2=\eM_1+\eM_2 \qquad \textrm{and} \qquad \eM_1\wedge \eM_2=\eM_1\cap \eM_2.$$
The linear span of a non-empty set $\mathcal S$ of vectors in $\eV$ is denoted by $\bigvee \mathcal S$. If $\mathcal S=\{x\}$, then we write $[x]$ instead of $\bigvee\{x\}$.
If $\mathcal S=\{x_1,\ldots,x_k\}$, then we clearly have
$\bigvee\{x_1,\ldots,x_k\}=[x_1]\vee \cdots \vee[x_k]=[x_1]+\cdots+[x_k]$.

By $\pL(\eV)$ we will denote the algebra
of all linear transformations on $\eV$. The kernel and the range of a linear transformation $A\in \pL(\eV)$ are denoted by $\eN(A)$ and $\eR(A)$, respectively. For a subalgebra $\pA\subseteq \pL(\eV)$ which contains the identity transformation $I$ we denote by $\pA^{-1}$
the group of all invertible linear transformations in $\pA$. For a given transformation $A\in \pL(\eV)$ we denote by $m_A$ and  $(A)=\{ p(A);\; p(z)\in \bC[z]\}$ the minimal polynomial of $A$ and the subalgebra of $\pL(\eV)$ which is generated by $A$ and $I$, respectively. If $x\in \eV$, then $(A)_x=\{ p(A)x;\;\; p(z)\in \bC[z]\}$ is called the cyclic subspace of $A$
generated by $x$. A subspace $\eM$ is invariant under a set $\mathcal S\subseteq \eL(\eV)$ if $\mathcal S\eM\subseteq\eM$. It is clear that $(A)_x$ is the smallest subspace which contains $x$ and is invariant under $A$. Since the sum and the intersection of invariant subspaces of $A\in \pL(\eV)$ are again invariant under $A$, the set $\Lat(A)$ of all invariant subspaces of $A$ is a sublattice of $\fL$.
It is obvious that $\fL=\Lat(I)$.
For $\eM_1,\eM_2\in \Lat(A)$, the set $[\eM_1,\eM_2]=\{\eP\in \Lat(A);\; \eM_1\leq \eP\leq \eM_2\}$ is called an interval in $\Lat(A)$. A subspace  $\eM\in \Lat(A)$ is a cycle if $[\{0\},\eM]$ is a finite chain in $\Lat(A)$.
For linear transformations of $A\in \pL(\eV)$ and $B\in \pL(\eW)$, we denote with $(A,B)^{\Int}$ the set of all
intertwiners of $A$ and $B$, that is, the set of all
 linear transformations $T\colon \pL(\eW)\to \pL(\eV)$ such that $AT=TB$. If $\eV=\eW$ and $A=B$, then $(A,A)^{\Int}=\{ T\in \pL(\eV);\; TA=AT\}$ is denoted by $(A)'$, and is called the commutant of $A$. A subspace $\eM$ of $\eV$ which is invariant under $(A)'$ is called a hyperinvariant subspace of $A$.

We conclude the preliminary section by introducing notation about complex matrices.
For $m, n\in \bN$ we denote by $\bC^n$ the space of all columns with $n$ rows and by $\bde_1,\ldots,\bde_n$ we denote its
standard basis vectors. By $\bM_{m\times n}$ we denote the space of all $m\times n$ complex matrices.
The zero matrix is denoted by $\bdo$, or by $\bdo_{m\times n}$ if we want to point out the dimension of the matrix.  The $m\times m$ identity matrix is denoted by $\bdI_m$ and the $m\times m$ Jordan block is denoted by $\bdJ_m$. It is easily seen that $(\bdJ_m)'=(\bdJ_m)$, that is, $(\bdJ_m)'$ consists of precisely all upper-triangular Toeplitz $m\times m$ matrices. The standard unit matrix $\bdE_{k,l}\in \bM_{m\times n}$ is the matrix whose $(k,l)$-th entry is $1$ while other entries are all zero. 

\section{General results about cyclic subspaces}\label{Section01}

In this section we consider some general results about collineations which will be needed throughout the text. In \Cref{prop08} we prove that $\Col(A)$ always contains invertible transformations of the commutant $(A)'$.

\begin{lemma} \label{lem20}
Let $T\in \pL(\eV)^{-1}$. Then $T\in \Col(A)$ if and only if $T(A)_x\in \Lat(A)$ for every $x\in \eV$.
\end{lemma}

\begin{proof}
If $T\in \Col(A)$, then $T(A)_x\in \Lat(A)$ for every $x\in \eV$ since every cyclic subspace $(A)_x$ is invariant for $A$.
To prove the opposite implication, assume that $\eM\in \Lat(A)$. If $\eM=\{0\}$, then it is clear that $T\eM=\eM\in \Lat(A)$. Suppose therefore that $\eM\ne \{ 0\}$ and let $\{e_1,\ldots, e_m\}$ be a basis of $\eM$.
Since $\eM=(A)_{e_1}\vee \cdots \vee (A)_{e_m}$ and $T(A)_{e_i}\in \Lat(A)$ for each $1\leq i\leq m$ we conclude that $T\eM=T(A)_{e_1}\vee\cdots\vee T(A)_{e_m}\in \Lat(A)$.
\end{proof}

\begin{proposition} \label{prop08}
Every invertible linear transformation in $(A)'$ is in $\Col(A)$.
\end{proposition}

\begin{proof}
Let $T\in (A)'$ be invertible. Since $T(A)_x=(A)_{Tx}\in \Lat(A)$ for every $x\in \eV$ it follows, by Lemma \ref{lem20},
that $T\in \Col(A)$.
\end{proof}

We continue with the study of cyclic subspaces of nilpotent linear transformations.
Let $N\in \pL(\eV)$ be a non-zero nilpotent linear transformation. The nil-index of $N$ is the smallest positive integer $n$ such that $N^{n-1}\neq 0$ and $N^{n}=0$.   Suppose that the nil-index $n$ of $N$ is at least $2$.  For a vector $0\ne x\in \eV$, let $0\leq k_x<n$ be the integer for which $N^{k_x}x\ne 0$ and $N^{k_x+1}x=0$. It is well known and easily seen that the vectors $x, Nx, \ldots, N^{k_x}x$ are linearly independent and so they form a basis of the cyclic subspace $(N)_x$. From $N^{k_x+1}x=0$ it follows that for each $0\leq j\leq k_x$ we have
$$(N)_{N^{k_x-j}x}=\bigvee\{ N^{k_x} x, N^{k_x-1}x,\ldots, N^{k_x-j}x\}.$$ In particular, we have
$(N)_{N^{k_x}x}=[N^{k_x} x]$ and $(N)_x=\bigvee\{ x, \ldots, N^{k_x} x\}$.
Since for each $j=0, 1, \ldots, k_x$ we have $\dim\bigl( (N)_{N^{k_x-j}x}\bigr)=j+1$,
the chain
\begin{equation} \label{eq64}
\{ 0\}=(N)_{N^{k_x+1} x}<(N)_{N^{k_x}x}<\cdots<(N)_{Nx}<(N)_x
\end{equation}
 is a chain of subspaces of $(N)_x$ of  the maximal possible length contained in the interval $[\{ 0\},(N)_x]$.
In the following lemma we prove that the interval $[\{ 0\},(N)_x]$ is equal to the chain \eqref{eq64} so that it is actually a cycle.

\begin{lemma} \label{lem60}
For every $x\in\eV$ the cyclic subspace $(N)_x$ is a cycle in $\Lat(N)$.
\end{lemma}

\begin{proof}
Since the case $x=0$ is trivial, we only need to consider the case when $x$ is non-zero. Due to the text preceding the lemma, it suffices to prove that every subspace in $(N)_x$ invariant under $N$ is an element of the chain \eqref{eq64}. Moreover, since every invariant subspace is built by cyclic subspaces,  it is enough to show that for every non-zero vector $y\in (N)_x$ the cyclic subspace $(N)_y$ is equal to one of the members of the chain \eqref{eq64}.

Let $y\in (N)_x$ be an arbitrary non-zero vector and write $y=\gamma_0 x+\gamma_1 Nx+\cdots+\gamma_{n_k} N^{k_x} x$ for some
some scalars $\gamma_0, \ldots, \gamma_{k_x}$. Since $y\neq 0$,  there exists the smallest index $j\geq 0$ such that $\gamma_j\ne 0$. From here it follows that
$y=(\gamma_j I+\gamma_{j+1}N+\cdots+ \gamma_{k_x} N^{k_x-j})N^j x\in (N)_{N^j x}$ which implies
$(N)_y\subseteq (N)_{N^j x}$. To prove the opposite inclusion, consider the linear transformation $q(N)=\gamma_j I+\gamma_{j+1}N+\cdots+ \gamma_{k_x} N^{k_x-j}$.
Since $\gamma_j\ne 0$, the transformation $q(N)$ is invertible in $(N)$, and so there exists $p(N)\in (N)$
such that $p(N)q(N)=I$. It follows that $p(N)y=N^j x$ and so $N^j x\in (N)_y$ giving also the opposite inclusion  $(N)_{N^j x}\subseteq (N)_y$.
\end{proof}

The following corollary immediately follows from \Cref{lem60}.

\begin{corollary} \label{lem63}
Let $x\in\eV$ be arbitrary. Suppose that $\eM_1, \eM_2\in \Lat(N)$ are such that $(N)_x=\eM_1\vee\eM_2$.
\begin{enumerate}
  \item [(i)] Then $\eM_1=(N)_x$ or $\eM_2=(N)_x$.
  \item [(ii)] If $(N)_x=\eM_1\oplus \eM_2$, then  $\eM_1=(N)_x$ and $\eM_2=\{ 0\}$ or $\eM_1=\{ 0\}$ and $\eM_2=(N)_x$.
\end{enumerate}
\end{corollary}

\begin{proof}
(i) Since $\eM_1, \eM_2\in [0,(N)_x]$, by \Cref{lem60} it follows that $\eM_1\subseteq \eM_2$ or $\eM_2\subseteq \eM_1$. While the former yields $(N)_x=\eM_2$, the latter yields $(N)_x=\eM_1$. The assertion (ii) follows from (i).
\end{proof}

\section{Reduction to the nilpotent case} \label{sec02}
\setcounter{theorem}{0}

Let $\eV$ and $\eW$ be isomorphic complex vector spaces and let $A\in \pL(\eV)$ and $B\in \pL(\eW)$ be similar linear transformations. Suppose that $B=SAS^{-1}$ where $S\in \pL(\eV,\eW)$ is invertible. Then $m_A=m_B$,  $(B)=S(A)S^{-1}$ and $(B)'=S(A)' S^{-1}$. It is well known that  $\eM\in \Lat(A)$ if and only if $S\eM\in \Lat(B)$, and similarly, $\eM\in \Lat (A)'$ if and only if $S\eM\in \Lat(B)'$. Since $S$ is invertible, the mapping $\Phi_S\colon \pL(\eV)\to \pL(\eW)$ given by $\Phi_S(\eM)=S\eM$ induces an isomorphism of lattices
of invariant subspaces $\Lat(A)$ and $\Lat(B)$, and similarly, of lattices of hyperinvariant subspaces $\Lat(A)'$ and $\Lat(B)'$. Therefore, the following result about groups of collineations of similar transformations should not be surprising.

\begin{lemma}
Groups of collineations satisfy $\Col(B)=S \Col(A) S^{-1}$.
\end{lemma}

\begin{proof}
  Suppose that $T\in \Col(A)$. Since every subspace $\eK\in\Lat(B)$ is of the form $\eK=S\eM$ for some
$\eM\in \Lat(A)$, we have $STS^{-1} \eK=ST\eM\in \Lat(B)$. Hence, $STS^{-1}\in \Col(B)$.
A similar reasoning shows that $S^{-1}T'S\in\Col(A)$ for every $T'\in \Col(B)$ which proves $\Col(B)=S\Col(A)S^{-1}$.
\end{proof}

Now we briefly outline the primary decomposition of a linear transformation.
Let $\lambda_1, \ldots,\lambda_s$ be pairwise distinct eigenvalues of a linear transformation $A\in \pL(\eV)$. 
Then the minimal polynomial $m_A$ of $A$ is of the form $m_A(z)=(z-\lambda_1)^{n_1}\cdots (z-\lambda_s)^{n_s}$
where $n_1, \ldots, n_s\in \bN$ and $n_1+\cdots+n_s\leq d=\dim \eV$.
For every $j=1, \ldots, s$, let  $\eV_j=\eN\bigl((A-\lambda_j I)^{n_j}\bigr)$ and let
$A_j$ be the restriction of $A$ to  $\eV_j$.
Then $A_j=\lambda_j I_{j}+N_j$, where $I_{j}$ is the identity on $\eV_j$ and $N_j\in \pL(\eV_j)$ is a nilpotent linear transformation with the nil-index $n_j$.
It follows that the minimal polynomial of $A_j$ is $m_{A_j}(z)=(z-\lambda_j)^{n_j}$.
The {\em primary decomposition} of $A$ is
\begin{equation} \label{eq05}
\eV=\eV_1 \oplus\cdots \oplus \eV_s,\qquad A=A_1\oplus \cdots \oplus A_s,\qquad A_j=\lambda_j I_{j}+N_j\;\;(j=1,\ldots,s).
\end{equation}
It should be noted that $(A_j)=(N_j)$, $(A_j)'=(N_j)'$, $\Lat(A_j)=\Lat(N_j)$ and $\Lat(A_j)'=\Lat(N_j)'$ for every $j=1,\ldots, s$.

\begin{proposition} \label{prop01}
Let $A\in \pL(\eV)$ and let \eqref{eq05} be its primary decomposition. Then
\begin{enumerate}
\item[(i)] $(A)=(A_1)\oplus \cdots\oplus (A_s)=(N_1)\oplus \cdots\oplus (N_s)$;
\item[(ii)] $(A)'=(A_1)'\oplus \cdots\oplus (A_s)'=(N_1)'\oplus \cdots\oplus (N_s)'$;
\item[(iii)] $\Lat(A)=\Lat(A_1)\oplus\cdots\oplus \Lat(A_s)=\Lat(N_1)\oplus\cdots\oplus \Lat(N_s)$;
\item[(iv)] $\Lat(A)'=\Lat(A_1)'\oplus\cdots\oplus \Lat(A_s)'=\Lat(N_1)'\oplus\cdots\oplus \Lat(N_s)'$.
\end{enumerate}
\end{proposition}

For the proofs of (ii), (iii) and (iv) we refer the reader to \cite[Lemma 1]{BF} and \cite[Theorem 2]{FHL}. We include the proof of (i) for reader's convenience  since we cannot find an exact reference.

\begin{proof}
The inclusion $(A)\subseteq (A_1)\oplus \cdots\oplus (A_s)$ is obvious. For the opposite inclusion note first that one only needs to consider the case $s=2$ as the general case follows by an easy induction argument on the number of primary factors. Choose linear transformations $B$ and $C$ which satisfy $\gcd\bigl(m_{B},m_{C}\bigr)=1$. Then there exist polynomials $r_1$ and $r_2$ such that $r_1m_{B}+r_2m_{C}=1$.
Pick an arbitrary $p_1(B)\oplus p_2(C)\in (B)\oplus (C)$ and consider the polynomial
$p=p_1r_2m_C+p_2r_1m_B$.
Then
\begin{align*}
p(B\oplus C)&=p(B)\oplus p(C)=\bigl(p_1(B)r_2(B)m_C(B)\bigr)\oplus \bigl(p_2(C)r_1(C)m_B(C)\bigr)\\
&=p_1(B)\oplus p_2(C)
\end{align*}
since $r_2(B)m_C(B)=I-r_1(B)m_B(B)=I$ and $r_1(C)m_B(C)=I-r_2(C)m_C(C)=I$.
This proves that $(B\oplus C)=(B)\oplus (C)$.
\end{proof}

Now we finally turn our attention to the reduction of the problem of characterizing $\Col(A)$ to the problem of characterizing the group of collineations of a given nilpotent linear transformation. The first step is to prove (see \Cref{prop04}) that a collineation $T\in \Col(A)$ permutes the subspaces from the primary decomposition of $A$. We start with a preparatory lemma involving cyclic subspaces.

\begin{lemma} \label{lem64}
Let $A\in \pL(\eV)$ and let \eqref{eq05} be its primary decomposition.
If $T\in \Col(A)$, then for every $0\ne x\in \eV_j$ there exists an index $i=\pi(x,j)$ such that $T(A)_x\leq \eV_i$.
\end{lemma}

\begin{proof}
Pick a non-zero vector $x\in \eV_j$ and denote $\widehat{N}_j=0\oplus\cdots\oplus N_j\oplus \cdots\oplus 0$. It is clear that $\widehat{N}_j\in \pL(\eV)$ is nilpotent and $(A)_x=(\widehat{N}_j)_x\subseteq \eV_j$. Since $T(\widehat{N}_j)_x
\in \Lat(A)$ and $\Lat(A)=\Lat(N_1)\oplus\cdots\oplus\Lat(N_s)$ by \Cref{prop01}, there exist subspaces $\eP_1\in \Lat(N_1)$, \ldots, $\eP_s\in \Lat(N_s)$ such that $T(\widehat{N}_j)_x=\eP_1\oplus \cdots\oplus \eP_s=\eP_1\vee \cdots\vee \eP_s$, where
in this last part each $\eP_l$ ($l=1,\ldots,s$) is considered as a subspace of $\eV$ in a natural way. Since $T^{-1}\in \Col(A)$
we have $(\widehat{N}_j)_x=T^{-1}\eP_1\vee \cdots\vee T^{-1}\eP_s$, where $T^{-1}\eP_l\in \Lat(A)$ for every $l=1,\ldots, s$. It follows by Lemma \ref{lem63} that
there is an index $i=\pi(x,j)$ such that $(\widehat{N}_j)_x=T^{-1}\eP_i$. We conclude that
$T(A)_x=T(\widehat{N}_j)_x=\eP_i\leq \eV_i$.
\end{proof}

\begin{proposition} \label{prop04}
Let $A\in \pL(\eV)$ and let \eqref{eq05} be its primary decomposition.
If $T\in \Col(A)$, then there exists a permutation $\pi$ of $\{ 1, \ldots, s\}$ such that $T\eV_j=\eV_{\pi(j)}$ for every $j=1, \ldots, s$.
\end{proposition}

\begin{proof}
Let $j\in \{ 1,\ldots, s\}$ and denote $\widehat{N}_j=0\oplus\cdots \oplus N_j \oplus \cdots\oplus 0$. Pick linearly independent vectors $x$ and $y\in \eV_j$. By Lemma \ref{lem64} there exist indices
$\pi(x,j)$ and $\pi(y,j)$ such that $T(A)_x\leq \eV_{\pi(x,j)}$ and $T(A)_y\leq \eV_{\pi(y,j)}$. Suppose, towards a
contradiction, that $\pi(x,j)\ne \pi(y,j)$. Then $\eV_{\pi(x,j)}\cap \eV_{\pi(y,j)}=\{ 0\}$, and therefore,
$T\widehat{N}_{j}^{m}x$ and $T\widehat{N}_{j}^{n}y$ are linearly independent for all indices $0\leq m\leq k_x$,
$0\leq n\leq k_y$. It follows that $\widehat{N}_{j}^{m}x$ and $\widehat{N}_{j}^{n}y$ are linearly independent.
Since $[\widehat{N}_{j}^{k_x}x]$ and $[\widehat{N}_{j}^{k_y}y]$ are one-dimensional subspaces in $\Lat(A)$,
$T[\widehat{N}_{j}^{k_x}x]$ and $T[\widehat{N}_{j}^{k_y}y]$ are one-dimensional subspaces in $\Lat(A)$ as well. Consider the non-zero vector $\widehat{N}_{j}^{k_x}x+\widehat{N}_{j}^{k_y}y\in \eV_j$ and note
that  $\widehat{N}_{j}(\widehat{N}_{j}^{k_x}x+\widehat{N}_{j}^{k_y}y)=0$. Then $[\widehat{N}_{j}^{k_x}x+\widehat{N}_{j}^{k_y}y]$ is an one-dimensional subspace in $\Lat(A)$ so that
$T[\widehat{N}_{j}^{k_x}x+\widehat{N}_{j}^{k_y}y]\in\Lat(A)$ is one-dimensional as well. Hence, there exists a vector $0\ne z\in \eV$
such that $T[\widehat{N}_{j}^{k_x}x+\widehat{N}_{j}^{k_y}y]=[z]$ and so $T\widehat{N}_{j}^{k_x}x+T\widehat{N}_{j}^{k_y}y=\lambda z$ for some number $\lambda\ne 0$. Thus, $[z]\leq T[\widehat{N}_{j}^{k_x}x]\oplus T[\widehat{N}_{j}^{k_y}y]$, and since $\Lat(A)=\Lat(N_1)\oplus\cdots\oplus\Lat(N_s)$ and $\dim[z]=1$, by \Cref{prop01}, it follows that
 either $[z]=T[\widehat{N}_{j}^{k_x}x]$ or $[z]=T[\widehat{N}_{j}^{k_y}y]$. If the former holds, then there exists a number $\mu\ne 0$ such that $T\widehat{N}_{j}^{k_x}x=\mu z$. From the identity
$T\widehat{N}_{j}^{k_x}x+T\widehat{N}_{j}^{k_y}y=\lambda z$ we conclude $T\widehat{N}_{j}^{k_y}y=(\lambda-\mu)z$ from where it follows that
$T\widehat{N}_{j}^{k_x}x$ and $T\widehat{N}_{j}^{k_y}y$ are linearly dependent which is a contradiction.
It is clear that we also get a contradiction if the latter holds.
We conclude that $\pi(x,j)=\pi(y,j)$ which means that there exists an index $\pi(j)\in \{ 1,\ldots, s\}$ such that $T\eV_j\subseteq \eV_{\pi(j)}$.

We claim that $T\eV_j=\eV_{\pi(j)}$. To this end, note first that a similar argument as above shows that there exists $\pi_1(j)$ such that $T^{-1}\eV_{\pi(j)}\subseteq \eV_{\pi_1(j)}$. Then
$\eV_j\subseteq T^{-1}\eV_{\pi(j)}\subseteq \eV_{\pi_1(j)}$ for each $1\leq j\leq s$  implies $\eV_j=\eV_{\pi_1(j)}=T^{-1}\eV_{\pi(j)}$ proving the claim.

To finish the proof note that invertibility of $T$ implies that $\pi\colon j\mapsto \pi(j)$ is a permutation of $\{1,\ldots, s\}$.
\end{proof}

In \Cref{prop04} we have observed that every $T\in \Col(A)$ permutes the subspaces from the primary decomposition of $A$. In \Cref{prop05} we consider the question for which subspaces $\eV_j$ and $\eV_k$ from the primary decomposition of $A$ there exists $T\in \Col(A)$ such that $T\eV_k=\eV_j$. The answer is connected to the lattice structure of $\Lat(A_k)$ and $\Lat(A_j)$.

\begin{proposition} \label{prop05}
Let $A\in \pL(\eV)$ and let \eqref{eq05} be its primary decomposition. Let $\pi$ be a permutation of $\{ 1, \ldots, s\}$.
For each index $k\in \{1, \ldots, s\}$ the following assertions are equivalent.
\begin{enumerate}
  \item[(i)] There exists $T \in \Col(A)$ such that $T\eV_k=\eV_{\pi(k)}$.
  \item[(ii)] Lattices $\Lat(A_k)$ and $\Lat(A_{\pi(k)})$ are isomorphic.
  \item[(iii)] Operators $N_k$ and $N_{\pi(k)}$ are similar.
\end{enumerate}
\end{proposition}

\begin{proof}
The equivalence between (ii) and (iii) follows from \cite[Corollary 2.3.1]{MM}.

To prove that (i) implies (ii), assume that $T\in \Col(A)$ satisfies $T\eV_k=\eV_{\pi(k)}$ and let $S\colon \eV_k\to \eV_{\pi(k)}$  be the
restriction of $T$ to $\eV_k$. 
It is clear that $S$ is an invertible linear transformation such that $S\eM\in \Lat(A_{\pi(k)})$ for every $\eM\in \Lat(A_k)$.  Since it also holds that for each $\eN\in \Lat(A_{\pi(k)})$ we have that $S^{-1}\eN\in \Lat(A_k)$, we conclude that
$\Lat(A_k)$ and $\Lat(A_{\pi(k)})$ are isomorphic.

To prove the opposite implication, suppose that $\Lat(A_k)$ and $\Lat(A_{\pi(k)})$ are isomorphic. By (iii) there exists an invertible linear bijection $S\colon \eV_k \to \eV_{\pi(k)}$ such that $SN_k=N_{\pi(k)}S$. It follows that $\eM\mapsto S\eM$ is an isomorphism of
lattices $\Lat(A_k)$ and $\Lat(A_{\pi(k)})$. Let $T\in \pL(\eV)$ be such linear transformation that the entries of its operator matrix
$\left[ T_{ij}\right]_{i,j=1}^{s}$ with respect to the decomposition $\eV=\eV_1\oplus \cdots\oplus \eV_s$ satisfy
$T_{\pi(k)k}=S$, $T_{k\pi(k)}=S^{-1}$, $T_{jj}=I_j$ for $j\notin\{k, \pi(k)\}$ and $T_{ij}=0$ otherwise. Then $T$ is invertible and $T\eV_k=\eV_{\pi(k)}$. To conclude that $T\in \Col(A)$ one needs to apply the splitting property of invariant subspaces (see \Cref{prop01}(i)) and the equality $SN_k=N_{\pi(k)}S$.
\end{proof}

Let $A\in \pL(\eV)$ and let \eqref{eq05} be its primary decomposition. We partition the set of indices
$\{ 1, \ldots, s\}$  into $1\leq t\leq s$ pairwise disjoint subsets as follows: indices $i,j$ are in the same set if and only
if $N_i$ and $N_j$ are similar. There is no loss of generality if we assume that our partition of $\{ 1,\ldots, s\}$ consists of subsets $\{ s_{l-1}+1, \ldots, s_{l}\}$
$(l=1,\ldots, t)$ where $0=s_0<s_1<\cdots<s_t=s$. Denote
\begin{equation} \label{eq78}
\eW_l=\eV_{s_{l-1}+1}\oplus\cdots\oplus \eV_{s_{l}}\quad \text{and}\quad
B_l=A_{s_{l-1}+1}\oplus\cdots\oplus A_{s_{l}}\qquad (l=1,\ldots, t).
\end{equation}
Hence, $\eV=\eW_1\oplus\cdots\oplus \eW_t$ and $A=B_1\oplus\cdots\oplus B_t$. Since $\sigma(B_i)\cap \sigma(B_j)=\emptyset$ for $i\neq j$ we have
$\Lat(A)=\Lat(B_1)\oplus\cdots\oplus \Lat(B_t)$.

\begin{proposition} \label{prop02}
Let $A\in \pL(\eV)$ and let \eqref{eq05} be its primary decomposition. If $B_1,\ldots,B_t$ are as in \eqref{eq78}, then
$\Col(A)=\Col(B_1)\oplus\cdots\oplus \Col(B_t)$.
\end{proposition}

\begin{proof}
Choose $T\in \Col(A)$ and let $\left[ T_{mn}\right]_{m,n=1}^{t}$ be its operator matrix with respect
to the decomposition $\eV=\eW_1\oplus\cdots\oplus \eW_t$. We will prove by way of contradiction that $T_{mn}=0$
for all indices $m\ne n$. If this were not the case, then there would exist indices $i\in \{ s_{m-1}+1,\ldots, s_m\}$, $j\in \{ s_{n-1}+1,\ldots, s_n\}$
and a vector $0\neq x\in \eV_j$ such that $0\ne Tx\in \eV_i$. By \Cref{prop04} we have $T(\eV_j)=\eV_i$, so that by \Cref{prop05} lattices
$\Lat(N_j)$ and $\Lat(N_i)$ are isomorphic. This contradiction shows that $T=T_{11}\oplus\cdots\oplus T_{tt}$. Since $T$ is invertible every $T_{jj}\in \pL(\eW_j)$ is invertible.
Let $\eM_j\in \Lat(B_j)$ $(j=1, \ldots,t)$ be
arbitrary. Then $\eM=\eM_1\oplus\cdots\oplus \eM_t\in \Lat(A)$ and therefore
$T\eM=T_{11}\eM_1\oplus\cdots\oplus T_{tt}\eM_t\in \Lat(A)$ which gives $T_{jj}\eM_j\in \Lat(B_j)$ for every
$j=1,\ldots, t$. It follows that $T_{jj}\in \Col(B_j)$.

To prove the opposite inclusion, suppose that $T_j\in \Col(B_j)$ for $j=1, \ldots t$. Let $T=T_1\oplus\cdots\oplus T_t$.
It is clear that $T\in \pL(\eV)$ is a bijection. If $\eM\in \Lat(A)$, then there exist $\eM_j\in \Lat(B_j)$
$(j=1, \ldots, t)$ such that $\eM=\eM_1\oplus\cdots\oplus\eM_t$. It follows that
$T\eM=T_1\eM_1\oplus\cdots\oplus T_t\eM_t\in \Lat(A)$. Hence, $T\in \Col(A)$.
\end{proof}

Now we consider a particular linear transformation $B$ of the form $B_l$ from \eqref{eq78}. Thus, let
\begin{equation} \label{eq79}
\eW=\eV_1\oplus \cdots \oplus \eV_r\quad \text{and}\quad B=A_1\oplus\cdots\oplus A_r,
\end{equation}
where $A_j=\lambda_j I_j+N_j$ ($j=1,\ldots, r$) for some distinct numbers $\lambda_1,\ldots,\lambda_r$ and
pairwise similar nilpotents $N_1,\ldots, N_r$. By \Cref{prop01} we have $\Lat(B)=\Lat(N_1)\oplus\cdots\oplus\Lat(N_r)$ and
lattices $\Lat(N_j)$ ($j=1,\ldots, r$) are pairwise isomorphic.
Let $\eV_0$ be any vector space isomorphic to the vector space $\eV_j$ $(j=1,\ldots,r)$ and let $N\in \pL(\eV_0)$ be such that $N$ is similar to
$N_j$ ($j=1,\ldots, r$). Let $S_j\colon \eV_0\to \eV_j$ be an invertible linear transformation such that
$N_j=S_j N S_{j}^{-1}$.

\begin{proposition} \label{prop06}
A linear transformation $T\in \pL(\eW)$ is in $\Col(B)$ if and only if there exists a permutation $\pi$ of $\{ 1,\ldots, r\}$ and collineations $T_1,\ldots, T_r\in \Col(N)$ such that the operator matrix of $T$ with respect to the decomposition
$\eW=\eV_1\oplus \cdots \oplus \eV_r$ is of the form $\left[ T_{ij}\right]_{i,j=1}^{r}$, where
$T_{\pi(k) k}=S_{\pi(k)}T_k S_{k}^{-1}$ for $k=1, \ldots, r$, and $T_{ij}=0$ if $(i,j)\ne (\pi(k),k)$ for every $k$.
\end{proposition}

\begin{proof}
Let $\pi$ be a permutation of $\{ 1,\ldots, r\}$ and let $T_1,\ldots, T_r\in \Col(N)$. It is easily seen that for every
$k=1, \ldots, r$ the linear transformation $S_{\pi(k)}T_k S_{k}^{-1}\colon \eV_k \to \eV_{\pi(k)}$ is invertible and that it induces a lattice isomorphism $\Phi_{S_{\pi(k)}T_k S_{k}^{-1}}\colon \Lat(N_k) \to \Lat(N_{\pi(k)})$. Since
$\Lat(B)=\Lat(N_1)\oplus \cdots\oplus \Lat(N_r)$ it follows that the linear transformation $T\in \pL(\eW)$ whose operator matrix with respect to the decomposition $\eW=\eV_1\oplus \cdots \oplus \eV_r$ is of the form
$\left[ T_{ij}\right]_{i,j=1}^{r}$ with $T_{\pi(k) k}=S_{\pi(k)}T_k S_{k}^{-1}$ for $k=1, \ldots, r$, and
$T_{ij}=0$ otherwise is in $\Col(B)$.

On the other hand, if $T\in \Col(B)$, then by Proposition \ref{prop04} there
exists a permutation $\pi$ of $\{ 1,\ldots, r\}$ such that $T\eV_k=\eV_{\pi(k)}$ for every $k=1, \ldots, r$. Let
$T_k\in \pL(\eV_k)$ be given by $T_k=S_{\pi(k)}^{-1}T|_{\eV_k} S_k$, where $T|_{\eV_k}$ is the restriction of
$T$ to $\eV_k$. It is clear that $T_k$ is a bijection. Let $\eM_k\in \Lat(N_k)$ and let $\eM_{\pi(k)}=T|_{\eV_k}\eM_k
\subseteq \eV_{\pi(k)}$. Since $\eM=\{ 0\}\oplus\cdots \oplus \eM_k\oplus\cdots\oplus \{ 0\}\in \Lat(B)$ it follows that $T\eM=\{ 0\}\oplus\cdots \oplus \eM_{\pi(k)}\oplus\cdots\oplus \{ 0\}$ is in  $ \Lat(B)$. Hence,
$\eM_{\pi(k)}\in \Lat(N_{\pi(k)})$. It is clear that this implies $T_k\in \Col(N_k)$. Let $\left[ T_{ij}\right]_{i,j=1}^{r}$
be the operator matrix of $T$ with respect to the decomposition $\eW=\eV_1\oplus \cdots \oplus \eV_r$. Then
$T_{\pi(k) k}=S_{\pi(k)}T_k S_{k}^{-1}$ for $k=1, \ldots, r$, and $T_{ij}=0$ if $(i,j)\ne (\pi(k),k)$ for every $k$.
\end{proof}

\section{Reflexive cover of the commutant of a linear transformation} \label{sec03}
\setcounter{theorem}{0}

Recall that by \Cref{prop08}, for every linear transformation $A\in \pL(\eV)$ the set of all invertible linear transformations of the commutant $(A)'$ of $A$ is always contained in $\Col(A)$. If $A$ is nilpotent, then $\Col(A)$ is always contained in a particular subalgebra of $\pL(\eV)$ containing $\Col(A)$ (see \Cref{prop07}). This algebra is the algebra $\Alg \Lat(A)'$ which coincides with the so-called reflexive cover $\Refl (A)'$ of the commutant of $A$.  In order to prove this result we need some preparation.

We start by recalling some basic notions and facts needed throughout this section.
For a non-empty family $\fF\subseteq \fL$ of subspaces it is easily seen that the set
$$\Alg\fF=\{ T\in \pL(\eV);\; T\eM\subseteq \eM\; \text{for every}\; \eM\in \fF\}$$ of linear transformations on $\eV$
is an algebra which contains the identity transformation $I$. On the other hand, for a non-empty subset $\pS\subseteq \pL(\eV)$,
the set $\Lat\,\pS=\bigcap_{T\in \pS}\Lat(T)$ is a sublattice of $\fL$.
By \cite[Proposition 22.3]{Con} we have $\fF\subseteq \Lat\,\Alg \fF$, $\pS\subseteq \Alg \Lat\,\pS$,
$\Alg \fF= \Alg\Lat\,\Alg \fF$ and $\Lat\,\pS= \Lat\,\Alg \Lat\,\pS$.

Let $\eV$ and $\eW$ be complex vector spaces and let $\pS$ a linear subspace of $\pL(\eV,\eW)$.
Reflexive cover of $\pS$ is
\begin{equation} \label{eq75}
\Refl \pS=\{ T\in \pL(\eV,\eW);\; \text{for every}\; x\in \eV\; \text{there exists}\; S_x\in \pS\; \text{such that}\; Tx=S_x x\}.
\end{equation}
It is clear that $\Refl \pS$ is a subspace of $\pL(\eV,\eW)$ which contains $\pS$. If $\Refl \pS=\pS$, then $\pS$ is said
to be a reflexive subspace of $\pL(\eV,\eW)$. It should be noted that this notion of reflexivity is not connected to the notion of reflexive Banach spaces.

\begin{lemma} \label{lem65}
For every subspace $\pS\subseteq \pL(\eV)$ we have $\Refl \pS\subseteq \Alg\Lat \pS$. If
$\pA\subseteq \pL(\eV)$ is an algebra such that $I\in \pA$, then $\Alg\Lat \pA=\Refl \pA$.
\end{lemma}

\begin{proof}
Let $\eM\in \Lat \pS$ and let $T\in \Refl \pS$. Let $x\in \eM$ be arbitrary. By the definition of the reflexive cover,
there exists $S_x\in \pS$ such that $Tx=S_x x$. Since $\eM$ is invariant for every linear transformation in $\pS$
we have $Tx\in \eM$. Hence, $T\in \Alg\Lat \pS$.

Let $\pA\subseteq \pL(\eV)$ be an algebra which contains $I$. Choose $T\in \Alg\Lat\,\pA$. If $x\in \eV$, then $\pA x\in \Lat\, \pA$. Hence, $T\pA x\subseteq \pA x$.
Since $I\in \pA$ we have $Tx\in \pA x$ which means that there exists $A_x\in \pA$ such that $Tx=A_x x$.
\end{proof}

\begin{lemma} \label{lem68}
Let $\pS\subseteq \pL(\eV,\eW)$ be a subspace and let $\pA\subseteq \pL(\eW)$, $\pB\subseteq \pL(\eV)$ be
subalgebras. If $\pS$ is a $\pB$-$\pA$-module, that is, $BSA\in \pS$ for arbitrary $A\in \pA$, $B\in \pB$ and $S\in \pS$,
then $\Refl \pA$ and $\Refl\pB$ are algebras and $\Refl \pS$ is a $\Refl \pB$-$\Refl\pA$-module.
\end{lemma}

\begin{proof}
Let $A\in \Refl \pA$, $B\in \Refl \pB$ and $T\in \Refl \pS$ be arbitrary. For $x\in \eV$, there exist
$C_x\in \pA$ such that $Ax=C_x x$, $S_{Ax}\in \pS$ such that $T(Ax)=S_{Ax}(Ax)$, and $D_{TAx}\in \pB$
such that $B(TAx)=D_{TAx}(TAx)$. Since $BTAx=D_{TAx}S_{Ax}C_x x$ and $D_{TAx}S_{Ax}C_x\in \pS$
we conclude that $BTA\in \Refl \pS$.
It is clear that a similar reasoning gives that $\Refl \pA$ and $\Refl \pB$ are algebras, and therefore, $\Refl \pS$ is
a $\Refl \pB$-$\Refl\pA$-module.
\end{proof}

Let $r\geq 2$ be an integer. For $j=1, \ldots, r$, let $\eV_j$, $\eW_j$  be complex vector spaces and let
$\pS_j\subseteq \pL(\eV_j,\eW_j)$ be subspaces. Let us denote $\eV=\eV_1\oplus\cdots\oplus \eV_r$, $\eW=\eW_1\oplus\cdots\oplus \eW_r$ and $\pS=\pS_1\oplus\cdots\oplus \pS_r$. The following lemma yields that $\Refl \pS$ decomposes into the direct sum of $\Refl \pS_j$ for $j=1,\ldots,r$.

\begin{lemma} \label{lem69}
$\Refl \pS=\Refl \pS_1\oplus\cdots\oplus \Refl \pS_r$.
\end{lemma}

\begin{proof}
Assume that $T\in \Refl \pS$ and let $\left[ T_{ij}\right]_{i,j=1}^{r}$ be the operator matrix of $T$ with respect to the
decompositions  $\eV=\eV_1\oplus\cdots\oplus \eV_r$ and $\eW=\eW_1\oplus\cdots\oplus \eW_r$.
Choose an arbitrary $x_j\in \eV_j$ and let $x=0\oplus\cdots\oplus x_j\oplus\cdots\oplus 0\in \eV$. By the definition
of the reflexive cover, there exists $S_x=S_{x}^{(1)}\oplus\cdots\oplus S_{x}^{(r)} \in \pS$ such that
$Tx=S_x x$. Since $Tx=T_{1j}x_j\oplus\cdots\oplus T_{rj}x_j$ and $S_x x=0\oplus\cdots\oplus S_{x}^{(j)} x_j\oplus\cdots \oplus 0$ we see that $T_{ij}=0$ if $i\ne j$ and $T_{jj}\in \Refl \pS_j$. Hence, $T\in
\Refl \pS_1\oplus\cdots\oplus \Refl \pS_r$.

To prove the opposite inclusion, pick $T=T_1\oplus\cdots\oplus T_r\in \Refl \pS_1\oplus\cdots\oplus \Refl \pS_r$ and $x=x_1\oplus \cdots \oplus x_r\in X$. Then for each $1\leq j\leq r$ there exists $S_j\in \mathcal S_j$ such that $T_jx_j=S_jx_j$. Since
\begin{align*}
(T_1\oplus \cdots \oplus T_r)(x_1\oplus \cdots \oplus x_r)&=(T_1x_1\oplus \cdots \oplus T_rx_r)=S_1x_1\oplus \cdots\oplus S_rx_r\\
&=
(S_1\oplus \cdots \oplus S_r)(x_1\oplus \cdots \oplus x_r)
\end{align*}
and since
$S_1\oplus\cdots\oplus S_r\in \pS$, we conclude that $T_1\oplus \cdots \oplus T_r\in  \Refl \pS$.
\end{proof}

Let $A\in \pL(\eV)$ and let \eqref{eq05} be its primary decomposition. By Lemma \ref{lem69} and Proposition \ref{prop01} we have
$$ \Refl (A)=\Refl(N_1)\oplus\cdots\oplus\Refl(N_s) \quad \textrm{and}\quad
\Refl (A)'=\Refl(N_1)'\oplus\cdots\oplus\Refl(N_s)'.$$
Hence, in order to study $\Refl (A)$ or $\Refl (A)'$ it is enough to consider only nilpotent linear transformations.
Since in this paper we need only the reflexive cover of the commutant of a linear transformation we will focus on
a description of $\Refl(N)'$, where $N$ is a nilpotent linear transformation. We will restrict ourselves to complex matrices.

\begin{proposition} \label{prop09}
Let $m,n\in \bN$. If $m>n$, then $\bdS\in (\bdJ_m,\bdJ_n)^\Int$ if and only if the block matrix of $\bdS$
is $\left[ \begin{smallmatrix} p(\bdJ_n)\\ \bdo\end{smallmatrix}\right]$ for some polynomial $p$. Similarly, if $m<n$, then $\bdS\in (\bdJ_m,\bdJ_n)^\Int$ if and only if the block matrix of $\bdS$
is $\left[  \bdo\; p(\bdJ_m)\right]$.
\end{proposition}

\begin{proof}
We only consider the case $m<n$ as the other case can be proved similarly. Since the matrix $\bdJ_n$ can be written as
$$ \left[ \begin{matrix}
\bdJ_{n-m} & \bdE_{n-m,1} \\
\bdo & \bdJ_{m}
\end{matrix}\right]$$
with respect to the decomposition $\bC^n=\bC^{n-m}\oplus \bC^{m}$, it
is easily seen that for $\bdS=\left[ \bdo\; p(\bdJ_m) \right]$ we have
$\bdJ_m\bdS=\bdS \bdJ_n$.

To prove the converse statement, let $\bdS\in \bM_{m\times n}$ be an arbitrary matrix such that $\bdJ_m \bdS=\bdS \bdJ_n$.
For every $j=0,1, \ldots, n-1$ we have $\bdJ_{n}^{j}\bde_n=\bde_{n-j}$ and so
$ \bdJ_{m}^{j}\bdS \bde_n=\bdS \bdJ_{n}^{j}\bde_n=\bdS \bde_{n-j}$.
Since $\bdJ_{m}^{j}=\bdo$ if $j\geq m$ we have $\bdS \bde_{i}=\bdo$ for $i=1, \ldots, n-m$. Hence, the block matrix
of $\bdS$ with respect to the decomposition $\bC^n=\bC^{n-m}\oplus \bC^m$ is of the form $\bdS=\left[ \bdo \; \bdB\right]$
for some $\bdB\in \bM_{m\times m}$.
It follows from  $\bdJ_m \bdS=\bdS \bdJ_n$ that $\bdJ_m\bdB=\bdB \bdJ_m$, and hence,  $\bdB=p(\bdJ_m)$ for some polynomial $p$.
\end{proof}

Let $\bdN\in \bM_{n\times n}$ be a nilpotent matrix. We may assume that $ \bdN=\bdJ_{n_{1}}\oplus\cdots\oplus \bdJ_{n_{k}}$, where  $n_{1}+\cdots+n_{k}=n$.
It is not hard to see that a matrix $\bdB\in \bM_{n\times n}$ commutes with $\bdN$ if and only if the block matrix
of $\bdB$ with respect to the decomposition $\bC^n=\bC^{n_{1}}\oplus \cdots\oplus \bC^{n_k}$ is
$\bdB=\left[ \bdB_{ij} \right]_{i,j=1}^{k}$, where $\bdB_{ij}\in (\bdJ_{n_i},\bdJ_{n_j})^\Int$. We can write this as
$ (\bdN)'=\bigl[(\bdJ_{n_i},\bdJ_{n_j})^\Int\bigr]_{i,j=1}^{k}$.

\begin{lemma} \label{lem67}
Reflexive cover of the commutant $(\bdN)'$ is $ \Refl(\bdN)'=\bigl[\Refl(\bdJ_{n_i},\bdJ_{n_j})^\Int\bigr]_{i,j=1}^{k}$.
\end{lemma}

\begin{proof}
Assume that $\bdT\in \Refl(\bdN)'$ and let $\left[ \bdT_{ij}\right]_{i,j=1}^{k}$ be its block matrix with respect to
the decomposition $\bC^n=\bC^{n_1}\oplus\cdots\oplus\bC^{n_k}$. Let $\bdx_j\in \bC^{n_j}$ be arbitrary
and let $\bdx=\bdo\oplus\cdots\oplus \bdx_j\oplus\cdots\oplus \bdo\in \bC^n$. By the definition of the reflexive
cover, there exists a matrix $\bdS_{\bdx}\in (\bdN)'$ such that $\bdT\bdx=\bdS_{\bdx} \bdx$. Let $\left[ \bdS_{\bdx}^{ij}\right]_{i,j=1}^{k}$ be the block matrix of $\bdS_{\bdx}$. Then $\bdS_{\bdx}^{ij}\in (\bdJ_{n_i},\bdJ_{n_j})^\Int$ for all indices $1\leq i, j\leq k$. Since
$\bdT \bdx=\bdT_{1j}\bdx_j\oplus\cdots\oplus \bdT_{kj}\bdx_j$ and
$\bdS_{\bdx} \bdx=\bdS_{\bdx}^{1j}\bdx_j\oplus\cdots\oplus \bdS_{\bdx}^{kj}\bdx_j$ we have
$\bdT_{ij}\bdx_j=\bdS_{\bdx}^{ij}\bdx_j$ for every $i=1, \ldots, k$ from where we conclude that $\bdT_{ij}\in \Refl(\bdJ_{n_i},\bdJ_{n_j})^\Int$ for all $1\leq i, j\leq k$.

To prove the opposite inclusion, suppose that $\bdT=\left[ \bdT_{ij}\right]_{i,j=1}^{k}$ is such that $\bdT_{ij}\in \Refl(\bdJ_{n_i},\bdJ_{n_j})^\Int$ for all $1\leq i,j\leq k$ and let $\bdx=\bdx_1\oplus\cdots\oplus\bdx_k$ be arbitrary.
There exist matrices $\bdS_{\bdx_j}^{ij}\in (\bdJ_{n_i},\bdJ_{n_j})^\Int$ such that $\bdT_{ij}\bdx_j=
\bdS_{\bdx_j}^{ij}\bdx_j$ ($1\leq i,j\leq k$). Let $\bdS_{\bdx}=\left[\bdS_{\bdx_j}^{ij}\right]_{i,j}^{k}$. Then
$\bdS_{\bdx}\in (\bdN)'$ and $\bdT \bdx=\bdS_{\bdx}\bdx$.
\end{proof}

\begin{proposition} \label{prop10}
Let $m,n\in \bN$. If $m \geq n$, then $\bdT\in \Refl (\bdJ_m,\bdJ_n)^\Int$ if and only if the block matrix of $\bdT$
is of the form $\left[ \begin{smallmatrix} \bdU\\ \bdo\end{smallmatrix}\right]$ for some upper-triangular matrix $\bdU\in \bM_{n\times n}$. Similarly, if $m\leq n$, then $\bdT\in \Refl(\bdJ_m,\bdJ_n)^\Int$ if and only if the block matrix of $\bdT$
is of the form $\left[  \bdo\; \bdV\right]$   for some upper-triangular matrix $\bdV\in \bM_{m\times m}$ .
\end{proposition}

\begin{proof}
We only consider the case $m\leq n$ as the other case can be proved similarly. Choose $T=\left[\tau_{ij}\right]\in \Refl(\bdJ_m,\bdJ_n)^\Int$. To prove that $\bdT$ is of the form $\left[\bdo\; \bdV\right]$ for some upper-triangular matrix $\bdV\in \bM_{m\times m}$, choose arbitrary indices $i\in \{ 1,\ldots,m\}$  and  $j\in \{1,\ldots, n\}$ with $j<n-m+i$ where  $\{\bde_1,\ldots,\bde_n\}$ and  $\{\bdf_1,\ldots,\bdf_m\}$ are the sets of  standard basis in $\bC^n$ and $\bC^m$, respectively. By the definition of the reflexive cover, there exists
$\bdS_{\bde_j}\in (\bdJ_m,\bdJ_n)^\Int$ such that $\bdT \bde_j=\bdS_{\bde_j}\bde_j$. Since by \Cref{prop09} a matrix $\bdS\in \bM_{m\times n}$ is in $(\bdJ_m,\bdJ_n)^\Int$ if and only if
its block matrix is $\left[  \bdo\; p(\bdJ_m)\right]$ for some polynomial $p(z)$ (if $m=n$, then actually
$\bdS=p(\bdJ_m)$) it follows that
$\tau_{ij}=\bdf_{i}^{*}\bdT \bde_j=\bdf_{i}^{*}\bdS_{\bde_j}\bde_j=0$.
Hence, the block matrix of $\bdT$ is
$\left[  \bdo\; \bdV\right]$, where $\bdV\in \bM_{m\times m}$ is an upper-triangular matrix.

Now we will prove that $\left[  \bdo\; \bdV\right] \in \Refl (\bdJ_m,\bdJ_n)^\Int$ for every upper-triangular
matrix $\bdV\in \bM_{m\times m}$. Since $\Refl (\bdJ_m,\bdJ_n)^\Int$ is a complex vector space it is enough to see that
for every pair of integers $(i,j)$ which satisfies $1\leq i\leq m$ and $i\leq j\leq m$ every matrix
$\bdE_{i,n-m+j}\in \bM_{m\times n}$
belongs to $ \Refl (\bdJ_m,\bdJ_n)^\Int$.
Let $j\in \{1,\ldots,m\}$. Then for an arbitrary vector $\bdx=x_1\bde_1+\cdots+x_n \bde_n\in \bC^n$ we have $\bdE_{j,n-m+j}\bdx=x_{n-m+j}\bdf_j$. If $x_{n-m+j}=0$, then $\bdE_{j,n-m+j}\bdx=\bdo_{m\times n}\bdx$. Assume therefore
that $x_{n-m+j}\ne 0$ and let $k$ be the largest integer in $\{1, \ldots, m\}$ such that $x_{n-m+k}\ne 0$.
Consider the following system of $k$ linear equations with $k$ variables $\xi_1,\ldots,\xi_k$:
\begin{equation} \label{eq80}
\begin{split}
x_{n-m+i}&\xi_1+\cdots+x_{n-m+k}\xi_{k-i+1}=0\qquad (1\leq i\leq k,\, i\ne j)\\
x_{n-m+j}&\xi_1+\cdots+x_{n-m+k}\xi_{k-j+1}=x_{n-m+j}.
\end{split}
\end{equation}
The matrix associated to system \eqref{eq80} is an anti-triangular Hankel matrix whose non-zero diagonal entries are $x_{n-m+k}$. Therefore, it is invertible from where it follows that system \eqref{eq80} has a unique solution $(s_1, \ldots, s_k)$. Let
$p(z)=s_1+s_2 z+\cdots+s_k z^{k-1}$ and $\bdS_{\bdx}=\left[ \bdo\; p(\bdJ_m)\right]\in (\bdJ_m,\bdJ_n)^\Int$.
It is clear that $\bdS_{\bdx}\bdx=\bdE_{j,n-m+j}\bdx$ which proves that $\bdE_{j,n-m+j}\in \Refl (\bdJ_m,\bdJ_n)^\Int$. Since $(\bdJ_m,\bdJ_n)^\Int$ is a right $(\bdJ_n)'$-module, by Lemma \ref{lem68},
$\Refl (\bdJ_m,\bdJ_n)^\Int$ is a right $(\bdJ_n)'$-module, as well. Since $\bdE_{j,n-m+j} \bdJ_{n}^{i}=
\bdE_{j,n-m+j+i}$ for $i=0,\ldots,m-j$ we conclude that $\bdE_{j,n-m+j+i}\in \Refl (\bdJ_m,\bdJ_n)^\Int$.
\end{proof}

The following corollary immediately follows from \Cref{prop10}.

\begin{corollary} \label{cor01}
Let $ \bdN=\bdJ_{n_{1}}\oplus\cdots\oplus \bdJ_{n_{k}}$. Then every diagonal $n\times n$ matrix is in
$\Refl(\bdN)'$.
\end{corollary}

\section{Colineations of a nilpotent linear transformation} \label{sec04}
\setcounter{theorem}{0}

In this section we consider $\Col(N)$ for a non-zero nilpotent linear transformation $N\in \pL(\eV)$ whose nil-index is $n\geq 2$. Recall that for $0\ne x\in \eV$ we denote by $k_x$ the largest integer such that $N^{k_x}x\ne 0$.  By Lemma \ref{lem60},
for every integer $j=0,\ldots, k_x$, the interval $[\{ 0\},(N)_{N^{k_x-j}x}]$ in the lattice $\Lat(N)$ is equal to the chain
$$ \{ 0\}<(N)_{N^{k_x} x}<\cdots<(N)_{N^{k_x-j}x}.$$
If $T\in \Col(N)$, then the mapping $\Phi_T\colon \Lat(N)\to \Lat(N)$ is a lattice isomorphism, and so $\Phi_T$ preserves intervals and chains of $\Lat(N)$. This implies that for each integer $j=0,\ldots, k_x$ the interval
$[\{ 0\},T(N)_{N^{k_x-j}x}]$ is equal to the chain
$$ \{ 0\}<T(N)_{N^{k_x} x}<\cdots<T(N)_{N^{k_x-j}x}.$$

\begin{theorem} \label{theo05}
If $T\in \Col(N)$ and $0\ne x\in \eV$, then for every $j=0,\ldots, k_x$ we have
\begin{equation} \label{eq66}
 T(N)_{N^{k_x-j}x}=(N)_{TN^{k_x-j}x}=(N)_{N^{k_x-j}Tx}.
\end{equation}
\end{theorem}

\begin{proof}
We will first prove by induction that $T(N)_{N^{k_x-j}x}=(N)_{TN^{k_x-j}x}$ for every $j=0,\ldots, k_x$. Let us first consider the case $j=0$.
Since $N^{k_x+1}x=0$ the cyclic subspace $(N)_{N^{k_x} x}$ is one-dimensional and so $(N)_{N^{k_x} x}=[N^{k_x} x]$. It follows that $T(N)_{N^{k_x} x}=T[ N^{k_x} x]=[ TN^{k_x} x]\in \Lat(N)$.
Hence, $N[ TN^{k_x} x]\subseteq [ TN^{k_x} x]$. Since $0$ is the only eigenvalue of $N$ we have
\begin{equation} \label{eq65}
NTN^{k_x} x=0.
\end{equation}
This gives that $(N)_{TN^{k_x} x}$ is one-dimensional. Since $[ TN^{k_x} x]\subseteq (N)_{TN^{k_x} x}$ and both subspaces
are one-dimensional we have $T(N)_{N^{k_x} x}=(N)_{TN^{k_x} x}$. This concludes the proof for $j=0$.

Assume now that $T(N)_{N^{k_x-j}x}=(N)_{TN^{k_x-j}x}$ for some index $0\leq j< k_x$. We claim that  $ (N)_{TN^{k_x-(j+1)}x}\subseteq T (N)_{N^{k_x-(j+1)}x}. $
From equalities
$$ (N)_{N^{k_x-(j+1)}x}=\bigvee\{ N^{k_x} x,\ldots,N^{k_x-j}x,N^{k_x-(j+1)}x\}=(N)_{N^{k_x-j}x}\vee [N^{k_x-(j+1)}x] $$
and $(N)_{N^{k_x-j}x}\land [N^{k_x-(j+1)}x]=\{ 0\}$ we conclude
\begin{equation}\label{eq5555}
T(N)_{N^{k_x-(j+1)}x}=T(N)_{N^{k_x-j}x}\vee T[N^{k_x-(j+1)}x]=(N)_{TN^{k_x-j}x}\vee [TN^{k_x-(j+1)}x]
\end{equation}
and $(N)_{TN^{k_x-j}x}\land [TN^{k_x-(j+1)}x]=\{ 0\}$. From $T(N)_{N^{k_x-(j+1)}x}\in \Lat(N)$ we first conclude
$ NTN^{k_x-(j+1)}x\in T(N)_{N^{k_x-(j+1)}x}$ and then by applying equality \eqref{eq5555} we argue that there exists a polynomial $p_1$ and a number $\gamma$
such that
$$ NTN^{k_x-(j+1)}x=p_1(N)TN^{k_x-j}x+\gamma TN^{k_x-(j+1)}x. $$
By an induction argument we can prove that for every positive integer $r$ we have
$$ N^rTN^{k_x-(j+1)}x=p_r(N)TN^{k_x-j}x+\gamma^r TN^{k_x-(j+1)}x$$ 
for some suitable polynomial $p_r$. If $r$ is greater than the nil-index of $N$, we have
$$p_r(N)TN^{k_x-j}x+\gamma^r TN^{k_x-(j+1)}x=0,$$
and since $0\ne TN^{k_x-(j+1)}x\not\in (N)_{TN^{k_x-j}x}$, we conclude $\gamma=0$. This yields
$NTN^{k_x-(j+1)}x\in (N)_{TN^{k_x-j}x}$ from where we conclude that
$$ N^rTN^{k_x-(j+1)}x\in (N)_{TN^{k_x-j}x}$$ 
for every $r\geq 1$
which means that
$$ (N)_{TN^{k_x-(j+1)}x}\subseteq (N)_{TN^{k_x-j}x}\vee [TN^{k_x-(j+1)}x].$$
Hence, for an arbitrary vector $q(N)TN^{k_x-(j+1)}x\in (N)_{TN^{k_x-(j+1)}x}$ we have
$$q(N)TN^{k_x-(j+1)}x=s(N)TN^{k_x-j}x+\delta TN^{k_x-(j+1)}x\in (N)_{TN^{k_x-j}x}\vee [TN^{k_x-(j+1)}x],$$
for some polynomial $s$ and some number $\delta$.
From $(N)_{TN^{k_x-j}x}=T(N)_{N^{k_x-j}x}$ we obtain $s(N)TN^{k_x-j}x=T t(N)N^{k_x-j}x$ for
some polynomial $t$, and so
\begin{equation*}
\begin{split}
q(N)&TN^{k_x-(j+1)}x=T t(N)N^{k_x-j}x+\delta TN^{k_x-(j+1)}x=\\
&=T\bigl(t(N)N^{k_x-j}x+\delta N^{k_x-(j+1)}x\bigr)\in
T\bigl( (N)_{N^{k_x-j}x}\vee [N^{k_x-(j+1)}x]\bigr)=T (N)_{N^{k_x-(j+1)}x}.
\end{split}
\end{equation*}
which proves the claim.

Since $ (N)_{TN^{k_x-(j+1)}x}\in \Lat(N)$ and $\Lat(N)$ is a cycle there exists an index $i\in\{ 0,\ldots, j+1\}$ such that
$ (N)_{TN^{k_x-(j+1)}x}=T (N)_{N^{k_x-i}x}$. In particular,
$TN^{k_x-(j+1)}x\in T (N)_{N^{k_x-i}x}$ which means that there exists a polynomial $p$ such that
$TN^{k_x-(j+1)}x=Tp(N)N^{k_x-i}x$, and consequently $N^{k_x-(j+1)}x=p(N)N^{k_x-i}x\in (N)_{N^{k_x-i}x}.$
Since $N^{k_x-(j+1)}x\not\in (N)_{N^{k_x-i}x}$ if $i<j+1$ we conclude that $i$ must be equal to $j+1$ proving the equality
$ (N)_{TN^{k_x-(j+1)}x}=T (N)_{N^{k_x-(j+1)}x}$. This proves $T(N)_{N^{k_x-j}x}=(N)_{TN^{k_x-j}x}$ for every $j=0,\ldots, k_x$.

To prove the equality $(N)_{TN^{k_x-j}x}=(N)_{N^{k_x-j}Tx}$ observe first that the first equality in \eqref{eq66} gives us  $T(N)_x=(N)_{Tx}$ for $j=k_x$. Therefore, the intervals $[\{0\},T(N)_x]$ and $[\{0\},(N)_{Tx}]$ are equal. By \Cref{lem60} they are chains of the same length and so by comparing the $j$-th $(j=0,\ldots, k_x)$ element of both chains we finally obtain $T(N)_{N^{k_x-j}x}=(N)_{N^{k_x-j}Tx}$.
\end{proof}

Next we prove that $\Col(N)\subseteq \Alg\Lat(N)'$.

\begin{proposition} \label{prop07}
If $T\in \Col(N)$, then $T\eM=\eM$ for every hyperinvariant subspace of $N$, that is, $\Col(N)\subseteq \Alg\Lat(N)'$.
\end{proposition}

\begin{proof}
By \cite[Theorem 1]{FHL}, $\Lat(N)'$ is a sublattice generated by those subspaces that are either the kernel or the range of an operator of the form $p(N)$ where $p$ is an arbitrary polynomial. Since every polynomial $p$ satisfies $p(z)=z^jq(z)$ for some integer $j\geq 0$ and some polynomial $q$ with $q(0)\neq 0$, from the fact that $q(N)$ is invertible, it follows that it suffices to see that
$T\eN(N^j)=\eN(N^j)$ and $T\eR(N^j)=\eR(N^j)$ for every integer $j\geq 0$.
If $j=0$ or $j\geq n$, then this trivially holds, so we only need to consider the case $1\leq j<n$.

Let $0\ne x\in \eN(N^j)$. Since $k_x$ is the largest integer such that $N^{k_x}x\ne 0$ we have $k_x<j$. The equality  $T(N)_{N^{k_x} x}=(N)_{N^{k_x} Tx}$ (see \Cref{theo05}) and formula \eqref{eq65} yield that $(N)_{N^{k_x} Tx}$ is a subspace of $\eN(N)$. In particular, we have $N^j Tx=0$ as $k_x+1\leq j$, so that $T\eN(N^j)\subseteq \eN(N^j)$. Since $T$ is invertible and $\eN(N^j)$ is finite-dimensional we have $T\eN(N^j)=\eN(N^j)$.

Let $0\ne x\in \eR(N^j)$ and let $y\in \eV$ be such that $x=N^j y$. It follows that $N^{k_x+j}y\ne 0$ and $N^{k_x+j+1}y=0$. Since $1\leq j\leq k_x+j$
we can apply \Cref{theo05} to get $T(N)_x=T(N)_{N^j y}=(N)_{N^j Ty}$. Hence, $Tx\in (N)_{N^j Ty}$
which means that there is a polynomial $p$ such that $Tx=p(N)N^jTy=N^jp(N)Tx\in \eR(N^j)$. As above it follows that
$T\eR(N^j)=\eR(N^j)$.
\end{proof}

\begin{corollary} \label{cor02}
If $T\in \Col(N)$, then for every $x\in \eV$ there exists $B_x\in (N)'$ such that $Tx=B_x x$.
\end{corollary}

\begin{proof}
Let $T\in \Col(N)$. Since for every $x\in \eV$ the subspace $(N)_{x}^{\prime}=\{ Bx;\; B\in (N)'\}$ is hyperinvariant for $N$, by Proposition \ref{prop07} we conclude that $T(N)_{x}^{\prime}=(N)_{x}^{\prime}$. Since $x\in (N)_{x}^{\prime}$ we
have $Tx=B_x x$ for some $B_x\in (N)'$.
\end{proof}

\begin{theorem} \label{theo07}
Let $T\in \pL(\eV)^{-1}$. Then $T\in \Col(N)$ if and only if
\begin{equation} \label{eq76}
\text{for every}\quad x\in \eV\quad \text{there exists}\quad C_x\in (N)'\quad \text{such that}\quad T(N)_x=C_x(N)_{x}.
\end{equation}
\end{theorem}

\begin{proof}
If $T\in \Col(N)$, then, by Corollary \ref{cor02}, for every $x\in \eV$ there exists $C_x\in (N)'$ such that $Tx=C_x x$. Hence, by Theorem \ref{theo05} we have
$$T(N)_x=(N)_{Tx}=(N)_{C_x x}=\{ p(N)C_x x;\; p\in \bC[z]\}=\{C_x p(N)x;\; p\in \bC[z]\}=
C_x(N)_x.$$ For the proof of the opposite implication observe first that for each $x\in \eV$ we have
$T(N)_x=C_x (N)_x=(N)_{C_x x}\in \Lat(N)$ and then apply \Cref{lem20}.
\end{proof}

It should be clear that for $T\in \pL(\eV)^{-1}$ condition \eqref{eq76} implies that
\begin{equation} \label{eq77}
\text{for every}\quad x\in \eV \quad \text{there exists}\quad B_x\in (N)'\quad \text{such that}\quad Tx=B_x x.
\end{equation}
Since $\Refl(N)'=\Alg\Lat(N)'$ by \Cref{lem65}, we obtain once more the set inclusion $\Col(N)\subseteq  \Alg\Lat(N)'$. The remaining part of this section is devoted to the characterization (see \Cref{theo01}) of those nilpotent linear transformations $N$ which satisfy $\Col(N)=\Alg\Lat(N)'$, that is, for which nilpotent linear transformations $N$ conditions \eqref{eq76} and \eqref{eq77} are equivalent.

Let $0\ne N\in \pL(\eV)$ be a nilpotent linear transformation and let $N|_{\eR(N)}$ be the restriction of $N$ to the range $\eR(N)$. It is easy to see that $\Lat(N|_{\eR(N)})\subseteq \Lat(N)$ and that for each $\eM\in \Lat(N|_{\eR(N)})$ the preimage $N^{-1}(\eM)$ is also invariant under $N$. In fact,  by \cite[Theorem 7]{BF},
\begin{equation} \label{eq51}
\Lat(N)=\bigcup_{\eM\in \Lat(N|_{\eR(N)})}\bigl[ \eM, N^{-1}(\eM)\bigr]
\end{equation}
where $[ \eM, N^{-1}(\eM)\bigr]$ is an interval in the lattice of all subspaces of $\eL(\eV)$.

\begin{proposition} \label{prop03}
If $T\in \Alg\Lat(N)'$ is invertible and $\eM$ is a hyperinvariant subspace of $N$, then $T\bigl([\eM,N^{-1}(\eM)]\bigr)=[\eM,N^{-1}(\eM)]$.
\end{proposition}

\begin{proof}
Since $T$ is invertible, $\Phi_T\colon \fL\to \fL$ is a lattice isomorphism. In particular, $\Phi_T$ preserves intervals of $\fL$. Hence, for each subspace $\eM$ of $\eV$ we have
$T\bigl([\eM,N^{-1}(\eM)]\bigr)=[T\eM,TN^{-1}(\eM)]$.
If, in addition,  $\eM$ is hyperinvariant for $N$, then $T\eM=\eM$ as $T\in \Alg\Lat(N)'$. To finish the proof we need to show that $N^{-1}(\eM)$ is also hyperinvariant for $N$. To this end, pick $x\in N^{-1}(\eM)$
and $C\in (N)'$. Then $NCx=CNx\in \eM$, so that $Cx\in N^{-1}(\eM)$.
\end{proof}

If $0\ne N\in \pL(\eV)$ is a nilpotent linear transformation such that every $\eM\in \Lat(N|_{\eR(N)})$ is hyperinvariant
for $N$, then it follows, by \eqref{eq51} and Proposition \ref{prop03}, that $T\in \Col(N)$ if  $T$ is invertible and $T\in \Alg\Lat(N)'$.

Let $T\in (N)'$. Then $\eR(N)$ is an invariant subspace for $T$, and therefore, the restriction $T|_{\eR(N)}$ of $T$ is a linear transformation on $\eR(N)$.
Since $T\in (N)'$, it should be clear that $T|_{\eR(N)}$ and $N|_{\eR(N)}$ commute. Hence, every hyperinvariant subspace for
$N|_{\eR(N)}$ is also hyperinvariant for $N$. By \cite[Theorem]{Ong} every
invariant subspace of $N|_{\eR(N)}$ is hyperinvariant for $N|_{\eR(N)}$ if and only if the Jordan decomposition
of $N|_{\eR(N)}$ consists of a single Jordan block.

Next we determine the group of invertible matrices in $\Col(N)'$ for nilpotent matrices $\bdJ_d$ and $\bdJ_k\oplus \bdo_l$.

\begin{example} \label{ex01}
(1) If $\bdN=\bdJ_d$, then $\eR(\bdN)=\bigvee\{ \bde_1,\ldots,\bde_{d-1}\}$ and $\bdN|_{\eR(\bdN)}=\bdJ_{d-1}$.
Since $(\bdJ_d)=(\bdJ_d)'$ we have $\Lat(\bdJ_d)=\Lat(\bdJ_d)'=\{\eM_j;\; 0\leq j\leq d\}$ where $\eM_0=\{0\}$ and
$\eM_j=\bigvee\{ \bde_1,\ldots,\bde_j\}$ for $1\leq j\leq d$. From this it easily follows that
$\Alg \Lat (\bdJ_d)'=\Alg \Lat (\bdJ_d)$ is precisely the algebra of all upper-triangular matrices. Thus, $\bdT\in \bM_{d\times d}$ is in
$\Col(\bdJ_d)$ if and only if it is an invertible upper-triangular matrix. Hence, $\Col(\bdJ_d)=\bigl(\Alg\Lat(\bdJ_d)'\bigr)^{-1}$.

(2) Let now $d\geq 3$ and let $k\geq 2, l\geq 1$ be such that $k+l=d$.
Denote $ \bdN=\bdJ_k\oplus \bdo_l$. Then it is clear that $(\bdN)=(\bdJ_k)\oplus (\bdo)$ and that
for an arbitrary vector $\bdx\oplus\bdy\in \bC^k\oplus\bC^l$ we have $\bdN(\bdx\oplus\bdy)=\bdJ_k\bdx\oplus \bdo$.
Hence, $\eN(\bdN)=\eN(\bdJ_k)\oplus \bC^l=\bigvee\{ \bde_1,\bde_{k+1},\ldots,\bde_d\}$ and $\eR(\bdN)=\eR(\bdJ_k)\oplus \{ \bdo\}=\bigvee\{ \bde_1,\ldots,\bde_{k-1}\}$.
Since $\bdN \bde_j=\bde_{j-1}$, for $j=2, \ldots, k-1$, and $\bdN \bde_1=\bdo$ we see that
$\bdN|_{\eR(\bdN)}=\bdJ_{k-1}$. Hence, the discussion preceding the example shows that $\Col(\bdN)$ is equal to the group of all invertible
matrices in $\Alg\Lat(\bdN)'$.

To see explicitly which matrices are in $\Col(\bdN)$, pick $\bdT\in \bM_{d\times d}$ and let $\bdT=\left[\begin{smallmatrix} \bdT_{11} & \bdT_{12}\\
\bdT_{21} & \bdT_{22} \end{smallmatrix}\right]$ be its block operator matrix with respect to the decomposition $\bC^d=\bC^k\oplus \bC^l$. A direct calculation shows that
$\bdT\in (\bdN)'$ if and only if
\begin{equation*}
\bdT_{11}=\left[ \begin{smallmatrix} \tau_1 & \tau_2 & \cdots & \tau_k\\ 0 & \tau_1 & \cdots & \tau_{k-1}\\
\vdots & \vdots  & \ddots & \vdots\\ 0 & 0 & \cdots & \tau_1
\end{smallmatrix}\right], \qquad
\bdT_{12}=\left[ \begin{smallmatrix} \kappa_1 & \kappa_2 & \cdots & \kappa_l\\ 0 & 0 & \cdots & 0\\
\vdots & \vdots  &   & \vdots\\ 0 & 0 & \cdots & 0
\end{smallmatrix}\right], \qquad
\bdT_{21}=\left[ \begin{smallmatrix} 0 & \cdots & 0 & \rho_1\\ 0 & \cdots & 0 & \rho_2\\
\vdots &   & \vdots   & \vdots\\ 0 & \cdots & 0 & \rho_l
\end{smallmatrix}\right]
\end{equation*}
and $\bdT_{22}\in \bM_{l\times l}$ is arbitrary. This implies (see Proposition \ref{prop10})
 that $\bdT\in\Col(\bdN)$ if and only if
\begin{equation*}
\bdT_{11}=\left[ \begin{smallmatrix} \tau_{11} & \tau_{12} & \cdots & \tau_{1k}\\ 0 & \tau_{22} & \cdots & \tau_{2\,k}\\
\vdots & \vdots  & \ddots & \vdots\\ 0 & 0 & \cdots & \tau_{kk}
\end{smallmatrix}\right], \;
\bdT_{12}=\left[ \begin{smallmatrix} \kappa_1 & \kappa_2 & \cdots & \kappa_l\\ 0 & 0 & \cdots & 0\\
\vdots & \vdots  &   & \vdots\\ 0 & 0 & \cdots & 0
\end{smallmatrix}\right], \;
\bdT_{21}=\left[ \begin{smallmatrix} 0 & \cdots & 0 & \rho_1\\ 0 & \cdots & 0 & \rho_2\\
\vdots &   & \vdots   & \vdots\\ 0 & \cdots & 0 & \rho_l
\end{smallmatrix}\right], \;
\bdT_{22}=\left[ \begin{smallmatrix} \theta_{11} & \theta_{12} & \cdots & \theta_{1l}\\ \theta_{21} & \theta_{22} & \cdots & \theta_{2\,l}\\
\vdots & \vdots  & \ddots & \vdots\\ \theta_{l1} & \theta_{l2} & \cdots & \theta_{ll}
\end{smallmatrix}\right]
\end{equation*}
and $\det(\bdT)=\tau_{11}\cdots\tau_{kk}\cdot\det(\bdT_{22})\ne 0$.
\end{example}

The situation completely changes whenever a given nilpotent matrix has more than one Jordan block of dimension at least two $2$.

\begin{theorem} \label{theo01}
Let $ \bdN=\bdJ_{n_{1}}\oplus\cdots\oplus \bdJ_{n_{k}}\in \bM_{n\times n}$ be a nilpotent matrix.
The group $\Col(\bdN)$ is a proper subgroup of $\bigl(\Alg\Lat(\bdN)'\bigr)^{-1}$ if and only if at least two Jordan
blocks are of dimension $2$ or more.
\end{theorem}

\begin{proof}
If $\bdN=\bdo$, then both $\Col(\bdN)$ and $ \bigl( \Alg\Lat(\bdN)'\bigr)^{-1}$ are equal to the group of all invertible $n\times n$ matrices. Assume therefore that $\bdN\ne \bdo$, and so, at least one Jordan block
is of dimension $2$ or more. If there is exactly one Jordan block of dimension $2$ or more, then \Cref{ex01} yields that $\Col(\bdN)=\bigl(\Alg\Lat(\bdN)'\bigr)^{-1}$.

Suppose now that there are at least two Jordan blocks in the Jordan decomposition of $N$ of dimension $2$ or more. Without loss of generality we
may assume that $n_1\geq 2$ and $n_2\geq 2$. 
Since the vectors $\bde_1$ and $\bde_{n_1+1}$ are in the kernel $\eN(\bdN)$, the one-dimensional space $\eM$ spanned by $\bde_1+\bde_{n_1+1}$ is invariant under $\bdN|_{\eR(\bdN)}$. 
Let us consider the one-dimensional space
$$\eK=\{ \lambda(\bde_1+\bde_{n_1+1})+\mu(\bde_2+\bde_{n_1+2});\; \lambda,\mu \in \bC\}$$
which obviously contains $\eM$.
From $\bdN\bigl( \lambda(\bde_1+\bde_{n_1+1})+\mu(\bde_2+\bde_{n_1+2})\bigr)=\mu(\bde_1+\bde_{n_1+1})\in \eM$ we conclude $\eK\in [\eM, \bdN^{-1}(\eM)$], and so $\eK\in \Lat(\bdN)$ by \eqref{eq51}.

Let $\bdD=\diag[\delta_1,\ldots,\delta_n]$ be an invertible diagonal matrix. By Corollary \ref{cor01} it follows that $\bdD\in \Alg\Lat(\bdN)'$.
In order to $\bdD\in \Col(\bdN)$, we will prove that $\bdD$ needs  to satisfy certain conditions which are not met by all invertible diagonal matrices. This will imply that $\Col(\bdN)$ is a proper subset of $\bigl(\Alg\Lat(\bdN)'\bigr)^{-1}$ and the proof will be finished.

Suppose $\bdD\in \Col(N)$. Then $\delta_j\ne 0$ for all $j=1, \ldots n$ as $\bdD$ is invertible.
Since $\bdD\in \Col(\bdN)$ the subspace
$$\bdD\eK = \{ \lambda(\delta_1 \bde_1+\delta_{n_1+1}\bde_{n_1+1})+\mu(\delta_2\bde_2+\delta_{n_1+2}\bde_{n_1+2});\; \lambda,\mu \in \bC\}$$
is invariant under $\bdN$. Therefore, there exist $\alpha,\beta \in \bC$ such that
\begin{align*}
\delta_2\bde_1+\delta_{n_1+2}\bde_{n_1+1}&=\bdN(\delta_2\bde_2+\delta_{n_1+2}\bde_{n_1+2})\\
&=\alpha(\delta_1\bde_1+\delta_{n_1+1}\bde_{n_1+1})+\beta(\delta_2\bde_2+\delta_{n_1+2}\bde_{n_1+2}).
\end{align*}
By comparing coefficients we obtain  $\delta_2=\alpha \delta_1$,
$\delta_{n_1+2}=\alpha \delta_{n_1+1}$ and $\beta \delta_2=\beta \delta_{n_1+2}=0$. Since numbers $\delta_j$
are non-zero we have $\beta=0$ and $\alpha=\frac{\delta_2}{\delta_1}=\frac{\delta_{n_1+2}}{\delta_{n_1+1}}$.
The last equality is clealy not fulfilled by every invertible diagonal matrix from $\Alg\Lat(\bdN)'$.
\end{proof}

\section{Colineations of $\bdN=\bdJ_2\oplus\bdJ_2\in \bM_{4\times 4}$} \label{sec05}
\setcounter{theorem}{0}

The simplest example of a nilpotent matrix $\bdN$ whose group of collineations is a proper subgroup of
$\bigl(\Alg\Lat(\bdN)'\bigr)^{-1}$ is $\bdN=\bdJ_2\oplus\bdJ_2\in \bM_{4\times 4}$. In this section we
determine $\Col(\bdN)$.

First, observe that
$$(\bdN)'=\left\{ \left[ \begin{smallmatrix}
\nu_{11} & \nu_{12} & \nu_{13} & \nu_{14} \\
0 & \nu_{11} & 0 & \nu_{13} \\
\nu_{31} & \nu_{32} & \nu_{33} & \nu_{34} \\
0 & \nu_{31} & 0 & \nu_{33}
\end{smallmatrix} \right];\; \nu_{ij}\in \bC\right\}\quad \text{and}\quad
\Alg\Lat(\bdN)'=\left\{ \left[ \begin{smallmatrix}
\gamma_{11} & \gamma_{12} & \gamma_{13} & \gamma_{14} \\
0 & \gamma_{22} & 0 & \gamma_{24} \\
\gamma_{31} & \gamma_{32} & \gamma_{33} & \gamma_{34} \\
0 & \gamma_{42} & 0 & \gamma_{44}
\end{smallmatrix} \right];\; \gamma_{ij}\in \bC\right\}$$
by \Cref{prop09} and \Cref{prop10}.
The kernel and the range of $\bdN$ coincide, and both are spanned by standard basis vectors $\bde_1$ and $\bde_3$.
Hence, $\bdN|_{\eR(\bdN)}=\bdo$ from where it follows that $\Lat(\bdN|_{\eR(\bdN)})$ consists of all subspaces of $\eR(\bdN)$. Since the range $\eR{(\bdN)}$ is two-dimensional the invariant subspaces of $\bdN|_{\eR(\bdN)}=\bdo$ are the trivial subspaces $\{ \bdo\}$ and $\eR(\bdN)$ and the family of one-dimensional subspaces
\begin{equation} \label{eq69}
\eM_{\omega,\kappa}=\left\{ \lambda(\omega \bde_1+\kappa\bde_3);\quad \lambda \in \bC\right\}\qquad\qquad (\omega,\kappa\in \bC,\quad|\omega|^2+|\kappa|^2\ne 0).
\end{equation}
Since $\bdN^{-1}(\{ \bdo\})=\eN(\bdN)$ and $\bdN^{-1}(\eR(\bdN))=\bC^4$
from  \eqref{eq51} we conclude
$$ \Lat(\bdN)=\left[ \{ \bdo\},\eN(\bdN)\right]\cup \left[ \eR(\bdN),\bC^4\right]\cup \biggl(\bigcup_{\omega,\kappa\in \bC\atop |\omega|^2+|\kappa|^2\ne 0}\left[\eM_{\omega,\kappa},\bdN^{-1}(\eM_{\omega,\kappa})\right]\biggr). $$
While the interval $\left[ \{ \bdo\},\eN(\bdN)\right]$ contains all subspaces of $ \eN(\bdN) $ the interval
$\left[ \eR(\bdN),\bC^4\right]$ contains all subspaces of $\bC^4$ which contain $\eR(\bdN)$.
If $\bdT\in \Alg\Lat(\bdN)'$ is invertible, then
$ \bdT\bigl( [\{ \bdo\},\eN(\bdN)]\bigr)=  [\{ \bdo\},\eN(\bdN)]$
and $\bdT\bigl([\eR(\bdN),\bC^4]\bigr)=[\eR(\bdN),\bC^4] $
as $\{\bdo\},\eN(\bdN), \eR(\bdN)$ and $\bC^4$ are hyperinvariant subspaces of $\bdN$.

Let $\omega,\kappa\in \bC$ be such that $|\omega|^2+|\kappa|^2\ne 0$.
By a direct calculation one can verify that the preimage $\bdN^{-1}(\eM_{\omega,\kappa})$ of $\eM_{\omega,\kappa}$ equals
$$\bdN^{-1}(\eM_{\omega,\kappa})=\left\{ \lambda \bde_1+\mu \bde_3+\zeta(\omega \bde_2+\kappa \bde_4);\;\lambda,\mu,\zeta\in \bC\right\}.$$
The subspaces from the interval $\left[\eM_{\omega,\kappa},\bdN^{-1}(\eM_{\omega,\kappa})\right]$ different from
$\eM_{\omega,\kappa}$ and $\bdN^{-1}(\eM_{\omega,\kappa})$ are two-dimensional. It is not hard
to see that each of them except the kernel $\eN(\bdN)$ is of the form
$$\eM_{\omega,\kappa}^{\rho,\theta}=\left\{ \lambda(\omega \bde_1+\kappa \bde_3)+\zeta(\rho \bde_1+\omega\bde_2+\theta \bde_3+\kappa\bde_4);\quad \lambda,\zeta\in \bC\right\}$$
for some $\rho$ and $\theta\in \bC$.
For each $\bdT=\left[ \begin{smallmatrix}
\gamma_{11} & \gamma_{12} & \gamma_{13} & \gamma_{14} \\
0 & \gamma_{22} & 0 & \gamma_{24} \\
\gamma_{31} & \gamma_{32} & \gamma_{33} & \gamma_{34} \\
0 & \gamma_{42} & 0 & \gamma_{44}
\end{smallmatrix} \right]\in \Alg\Lat(\bdN)'$ we have $$\det(\bdT)=(\gamma_{11}\gamma_{33}-\gamma_{13}\gamma_{31})(\gamma_{22}\gamma_{44}-\gamma_{24}\gamma_{42}),$$
and furthemore,  if $\bdT$ is invertble, then 
\begin{equation} \label{eq68}
 \gamma_{11}\gamma_{33}-\gamma_{13}\gamma_{31}\ne 0\quad \text{and}\quad
\gamma_{22}\gamma_{44}-\gamma_{24}\gamma_{42}\ne 0.
\end{equation}

\begin{theorem} \label{theo02}
Matrix $\bdT\in \bM_{4\times 4}$ is in $\Col(\bdN)$ if and only if
$$\bdT=\left[ \begin{smallmatrix}
\gamma_{11} & \gamma_{12} & \gamma_{13} & \gamma_{14} \\
0 & t\gamma_{11} & 0 & t\gamma_{13} \\
\gamma_{31} & \gamma_{32} & \gamma_{33} & \gamma_{34} \\
0 & t\gamma_{31} & 0 & t\gamma_{33}
\end{smallmatrix} \right],$$ where $t$ and $\gamma_{ij}$ are complex numbers such that
$ t(\gamma_{11}\gamma_{33}-\gamma_{13}\gamma_{31})\ne 0$.
\end{theorem}

\begin{proof}
Assume that $\bdT\in \Col(\bdN)\subseteq \bigl(\Alg\Lat(\bdN)'\bigr)^{-1}$. Then $\bdT$ is of the form
$$\bdT=\left[ \begin{smallmatrix}
\gamma_{11} & \gamma_{12} & \gamma_{13} & \gamma_{14} \\
0 & \gamma_{22} & 0 & \gamma_{24} \\
\gamma_{31} & \gamma_{32} & \gamma_{33} & \gamma_{34} \\
0 & \gamma_{42} & 0 & \gamma_{44}
\end{smallmatrix} \right]$$
with entries satisfying  $\eqref{eq68}$.

Let $(\omega,\kappa)\ne (0,0)$ be arbitrary and consider the subspace
$$\eM_{\omega,\kappa}^{0,0}=\{ \lambda(\omega \bde_1+\kappa \bde_3)+\zeta(\omega\bde_2+\kappa\bde_4);\;  \lambda,\zeta\in \bC\}\in \Lat(\bdN).$$
Since $\bdT\in \Col(\bdN)$, the subspace $\bdT \eM_{\omega,\kappa}^{0,0}\in \Lat(\bdN)$ is spanned by vectors
$\bdT(\omega \bde_1+\kappa \bde_3)=(\gamma_{11}\omega+\gamma_{13}\kappa)\bde_1+(\gamma_{31}\omega+\gamma_{33}\kappa)\bde_3$ and
$\bdT(\omega \bde_2+\kappa \bde_4)=(\gamma_{12}\omega+\gamma_{14}\kappa)\bde_1+(\gamma_{22}\omega+\gamma_{24}\kappa)\bde_2+(\gamma_{32}\omega+\gamma_{34}\kappa)\bde_3+(\gamma_{42}\omega+\gamma_{44}\kappa)\bde_4$. Since $\bdN \bdT(\omega \bde_2+\kappa \bde_4)$ is in
 $\bdT \eM_{\omega,\kappa}^{0,0}$ there exist numbers $\alpha, \beta\in \bC$ such that
$\bdN \bdT(\omega \bde_2+\kappa \bde_4)=\alpha\bdT(\omega \bde_1+\kappa \bde_3)+\beta\bdT(\omega \bde_2+\kappa \bde_4)$. It follows that
\begin{equation}\label{eq72}
\begin{split}
\gamma_{22}\omega+\gamma_{24}\kappa&=\alpha(\gamma_{11}\omega+\gamma_{13}\kappa)+\beta(\gamma_{12}\omega+\gamma_{14}\kappa)\\
0&=\beta(\gamma_{22}\omega+\gamma_{24}\kappa)\\
\gamma_{42}\omega+\gamma_{44}\kappa&=\alpha(\gamma_{31}\omega+\gamma_{33}\kappa)+\beta(\gamma_{32}\omega+\gamma_{34}\kappa)\\
0&=\beta(\gamma_{42}\omega+\gamma_{44}\kappa).
\end{split}
\end{equation}
If both numbers $\gamma_{22}\omega+\gamma_{24}\kappa$ and $\gamma_{42}\omega+\gamma_{44}\kappa$
were zero, then we would have $\gamma_{22}\gamma_{44}-\gamma_{24}\gamma_{42}=0$ which is impossible. Hence, $\gamma_{22}\omega+\gamma_{24}\kappa \ne 0$ or $\gamma_{42}\omega+\gamma_{44}\kappa\ne 0$. In any case, number $\beta$ in system of equations \eqref{eq72} has to be zero.
Therefore, the system \eqref{eq72} reduces to
\begin{equation}\label{eq73}
\begin{split}
\gamma_{22}\omega+\gamma_{24}\kappa&=\alpha(\gamma_{11}\omega+\gamma_{13}\kappa)\\
\gamma_{42}\omega+\gamma_{44}\kappa&=\alpha(\gamma_{31}\omega+\gamma_{33}\kappa).
\end{split}
\end{equation}

For arbitrary pair of numbers $(a, b)\in \bC^2\setminus \{(0,0)\}$ let
$$\Omega(a,b)=\{ (\omega,\kappa)\in \bC^2;\; a\omega+b\kappa\ne 0\}$$
be the complement of the line in $\bC^2$ given by the equation $a\xi+b\eta=0$.
If $(\omega,\kappa)\ne (0,0)$ is such that $\gamma_{11}\omega+\gamma_{13}\kappa=0$, then the first equality
in \eqref{eq73} gives $\gamma_{22}\omega+\gamma_{24}\kappa=0$ implying $\Omega(\gamma_{11},\gamma_{13})^c\subseteq \Omega(\gamma_{22},\gamma_{24})^c$. Since both sets are lines we conclude that
they are equal. Since $(\gamma_{11},\gamma_{13})$ and $(\gamma_{22},\gamma_{24})$ determine the same line, there exists a non-zero number $u$ such that
\begin{equation} \label{eq74}
\gamma_{22}=u\gamma_{11}\qquad \text{and}\qquad \gamma_{24}=u\gamma_{13}.
\end{equation}
Similarly one can show that  that there exists $v\ne 0$ such that
\begin{equation} \label{eq75druga}
\gamma_{42}=v\gamma_{31}\qquad \text{and}\qquad \gamma_{44}=v\gamma_{33}.
\end{equation}

Let $(\omega,\kappa)\in \Omega(\gamma_{11},\gamma_{13})\cap \Omega(\gamma_{31},\gamma_{33}) $. It follows from \eqref{eq73} that
$$(\gamma_{11}\gamma_{42}-\gamma_{22}\gamma_{31})\omega^2+(\gamma_{11}\gamma_{44}+\gamma_{13}\gamma_{42}-\gamma_{22}\gamma_{33}-\gamma_{24}\gamma_{31})\omega\kappa+(\gamma_{13}\gamma_{44}-\gamma_{24}\gamma_{33})\kappa^2=0.$$
Since the intersection $\Omega(\gamma_{11},\gamma_{13})\cap \Omega(\gamma_{31},\gamma_{33})$ is an open and dense subset of $\bC^2$  we conclude that
\begin{equation} \label{eq76enacba}
\gamma_{11}\gamma_{42}-\gamma_{22}\gamma_{31}=0,\qquad \qquad \gamma_{13}\gamma_{44}-\gamma_{24}\gamma_{33}=0
\end{equation}
and
\begin{equation} \label{eq77enacba}
\gamma_{11}\gamma_{44}+\gamma_{13}\gamma_{42}-\gamma_{22}\gamma_{33}-\gamma_{24}\gamma_{31}=0.
\end{equation}

To prove that $T$ is of the desired form, we need to consider different cases. If $\gamma_{11}, \gamma_{13}, \gamma_{31}$ and $\gamma_{33}$ are non-zero numbers, then it follows from
\eqref{eq74}, \eqref{eq75druga} and \eqref{eq76enacba} that $u=v$. Hence,
$\bdT=\left[ \begin{smallmatrix}
\gamma_{11} & \gamma_{12} & \gamma_{13} & \gamma_{14} \\
0 & u\gamma_{11} & 0 & u\gamma_{13} \\
\gamma_{31} & \gamma_{32} & \gamma_{33} & \gamma_{34} \\
0 & u\gamma_{31} & 0 & u\gamma_{33}
\end{smallmatrix} \right]$ in this case.

If $\gamma_{11}=0$ and $\gamma_{33}\ne 0$, then, by \eqref{eq74}, $\gamma_{22}=0$ and $\gamma_{44}\ne 0$.  It follows, by \eqref{eq68}, that $\gamma_{13}\ne 0$ and $\gamma_{31}\ne 0$. Hence,
$v=\frac{\gamma_{42}}{\gamma_{31}}=\frac{\gamma_{44}}{\gamma_{33}}=\frac{\gamma_{24}}{\gamma_{13}}=u$ and so
$$\bdT=\left[ \begin{smallmatrix}
0 & \gamma_{12} & \gamma_{13} & \gamma_{14} \\
0 & 0 & 0 & u\gamma_{13} \\
\gamma_{31} & \gamma_{32} & \gamma_{33} & \gamma_{34} \\
0 & u\gamma_{31} & 0 & u\gamma_{33}
\end{smallmatrix} \right].$$

If $\gamma_{11}=0$ and $\gamma_{33}= 0$, then $\gamma_{22}=0$ and $\gamma_{44}= 0$. By \eqref{eq68}, $\gamma_{13}\ne 0$ and $\gamma_{31}\ne 0$ so that \eqref{eq77enacba} reads as
$\gamma_{13}\gamma_{42}-\gamma_{24}\gamma_{31}=0$. Together with \eqref{eq74} and \eqref{eq75druga}
this gives $u=v$ and so $\bdT$ is of the form
$$\bdT=\left[ \begin{smallmatrix}
0 & \gamma_{12} & \gamma_{13} & \gamma_{14} \\
0 & 0 & 0 & u\gamma_{13} \\
\gamma_{31} & \gamma_{32} & 0 & \gamma_{34} \\
0 & u\gamma_{31} & 0 & 0
\end{smallmatrix} \right].$$

Following the same reasoning in the remaining cases we conclude that $\bdT\in\Col(\bdN)$ is of the form
$$\bdT=\left[\begin{smallmatrix}
\gamma_{11} & \gamma_{12} & \gamma_{13} & \gamma_{14} \\
0 & t\gamma_{11} & 0 & t\gamma_{13} \\
\gamma_{31} & \gamma_{32} & \gamma_{33} & \gamma_{34} \\
0 & t\gamma_{31} & 0 & t\gamma_{33}
\end{smallmatrix} \right],$$
where $t, \gamma_{11}, \gamma_{12}, \gamma_{13}, \gamma_{14},
\gamma_{31}, \gamma_{32}, \gamma_{33}, \gamma_{34} \in \bC$ are such that
$\det(\bdT)=\bigl(t(\gamma_{11}\gamma_{33}-\gamma_{13}\gamma_{31})\bigr)^2\ne 0$.

To prove the opposite implication, let
$$\bdT=\left[ \begin{smallmatrix}
\gamma_{11} & \gamma_{12} & \gamma_{13} & \gamma_{14} \\
0 & t\gamma_{11} & 0 & t\gamma_{13} \\
\gamma_{31} & \gamma_{32} & \gamma_{33} & \gamma_{34} \\
0 & t\gamma_{31} & 0 & t\gamma_{33}
\end{smallmatrix} \right]$$
be such that $t(\gamma_{11}\gamma_{33}-\gamma_{13}\gamma_{31})\ne 0$.
Then $\bdT\in \bigl(\Alg\Lat(\bdN)'\bigr)^{-1}$. We have  to see that $\bdT \eM_{\omega,\kappa}^{\rho,\theta}\in \Lat(\bdN)$ for all $\omega,\kappa,\rho,\theta\in \bC$ with $(\omega,\kappa)\ne (0,0)$. Since
$$ \left[ \begin{smallmatrix}
\gamma_{11} & \gamma_{12} & \gamma_{13} & \gamma_{14} \\
0 & t\gamma_{11} & 0 & t\gamma_{13} \\
\gamma_{31} & \gamma_{32} & \gamma_{33} & \gamma_{34} \\
0 & t\gamma_{31} & 0 & t\gamma_{33}
\end{smallmatrix} \right] \left[\begin{smallmatrix} \lambda\omega+\zeta\rho\\ \zeta\omega\\ \lambda\kappa+\zeta\theta\\ \zeta\kappa\end{smallmatrix}\right] =
\left[\begin{smallmatrix} \lambda(\gamma_{11}\omega+\gamma_{13}\kappa)+\zeta(\gamma_{11}\rho+\gamma_{12}\omega+\gamma_{13}\theta+\gamma_{14}\kappa)\\ t\zeta(\gamma_{11}\omega+\gamma_{13}\kappa)\\ \lambda(\gamma_{31}\omega+\gamma_{33}\kappa)+\zeta(\gamma_{31}\rho+\gamma_{32}\omega+\gamma_{33}\theta+\gamma_{34}\kappa)\\ tz\zeta(\gamma_{31}\omega+\gamma_{33}\kappa)\end{smallmatrix}\right]$$
for all $\lambda,\zeta\in \bC$ we have
$\bdT \eM_{\omega,\kappa}^{\rho,\theta}=\eM_{\gamma_{11}\omega+\gamma_{13}\kappa,\gamma_{31}\omega+\gamma_{33}\kappa}^{\gamma_{11}\rho+\gamma_{12}\omega+\gamma_{13}\theta+\gamma_{14}\kappa,\gamma_{31}\rho+\gamma_{32}\omega+\gamma_{33}\theta+\gamma_{34}\kappa}\in \Lat(\bdN)$.
\end{proof}

{\it Acknowledgments.}
The paper is a part of the project {\em Distinguished subspaces of a linear operator} and the work of the first author
was partially supported by the Slovenian Research Agency through the research program P2-0268. 
The second author acknowledges financial support from the Slovenian Research Agency, Grants No. P1-0222, J1-2453 and J1-2454. 


\end{document}